\documentclass[11pt]{article}
\usepackage[colorlinks=true,pdfborder=001,linkcolor=blue,urlcolor=blue,citecolor=blue]{hyperref}
\usepackage{graphicx}
\usepackage{float}
\usepackage{interval}
\usepackage{url}
\usepackage{bbm}
\usepackage{stackengine}
\usepackage{nccmath}
\usepackage{mathtools}

\usepackage{mathrsfs}
\usepackage[title]{appendix}
\usepackage{adjustbox}
\usepackage{subcaption}
\usepackage{pgfplots}
\usepackage{tikz}
\usepackage{tikz,fullpage}
\usepackage{rotating}
\usepackage{tabu}
\usepackage{hhline}
\usepackage{multirow}
\usepackage{float}
\restylefloat{table}
\usepackage{adjustbox}
\usepackage[round]{natbib}
\usepackage{booktabs}
\usepackage{array}
\usepackage{tabularx}
\usepackage[utf8]{inputenc}
\usepackage{amssymb,amsmath}
\usepackage{amsfonts}
\usepackage{latexsym}
\usepackage[english]{babel}
\usepackage[T1]{fontenc}
\usepackage{eepic}
\usepackage{epic}
\usepackage{epsfig}
\usepackage{graphics,color}
\usepackage{verbatim}
\usepackage{lscape}
\usepackage{amsthm}
\usepackage{amsmath}
\usepackage{placeins}
\usepackage{float}
\usepackage{caption}
\usepackage{color}
\usepackage{setspace}
\usepackage[ruled,linesnumbered]{algorithm2e}
\usepackage{frcursive}

\usepackage{mathrsfs}
\usepackage{soul}
\usepackage{tikz}

\usepackage{tikz,fullpage}
\usetikzlibrary{arrows,petri,topaths,automata}

\makeatletter

\renewcommand*\thesection{\arabic{section}}

\newtheorem{prop}{Proposition}

\newtheorem{exam}{Example}
\newtheorem*{exam*}{Example}

\newif\ifbold
\boldfalse
\boldtrue
\newcommand{\bbf}{\ifbold\bgroup\bf\fi}
\newcommand{\ebf}{\ifbold\egroup\fi}

\renewcommand{\textbf}[1]{\begingroup\bfseries\mathversion{bold}#1\endgroup}
\makeatletter    
\renewcommand{\section}{\@startsection {section}{1}{\z@}%
             {-2ex \@plus -1ex \@minus -.2ex}%
             {1ex \@plus.2ex}%
             {\normalfont\Large\rmfamily\bfseries}}
\renewcommand{\subsection}{\@startsection{subsection}{2}{\z@}%
             {-1.25ex\@plus -1ex \@minus -.2ex}%
             {.75ex \@plus .2ex}%
             {\normalfont\large\rmfamily\bfseries}}
\def\@listI{\leftmargin\leftmargini       
            \parsep .25ex \@plus .1ex     
            \topsep .25ex \@plus .1ex     
            \itemsep \parsep}
\let\@listi\@listI
\@listi
\makeatother
\makeatletter

\@addtoreset{equation}{section}
\makeatother
\makeatletter

\@addtoreset{table}{section}
\makeatother

\usepackage[margin=1in]{geometry}
\usepackage{graphicx}
\graphicspath{ {./Figures/} }
\definecolor{purple}{rgb}{0.4,0.2,1}
\usepackage{titling}

\title{
\LARGE\bf Scheduling and Routing in the Flexible Job Shop with Heterogeneous Transbots and Zoning: A Constraint Programming Approach
\vspace{1ex}
}

\author{\large Arnovi Moinuddin,$^{1,*}$ El Mehdi Er Raqabi,$^{1,2}$ Pascal Van Hentenryck$^{1}$ \\
\footnotesize$^1$\emph{H. Milton Stewart School of Industrial and Systems Engineering, Georgia Institute of Technology, Atlanta, USA}\\
\footnotesize$^2$\emph{Department of Operations and Decision Systems, Université Laval, Québec, Canada}\\
\footnotesize$^*$\emph{Corresponding Author: arnovi.moinuddin@gatech.edu}\\
}

\date{}

\begin{document}
\maketitle

\vspace{2cm}
\begin{abstract}
\vspace{0.5cm}
 Coordinating production and material transfers is increasingly important in modern manufacturing systems equipped with mobile transfer robots, known as transbots. This study considers a flexible job shop environment in which heterogeneous transbots transport parts between machines. The shop floor is partitioned into zones, with each transbot assigned to a specific zone, and inter-zone movements are facilitated through designated handoff points. These zoning constraints, transbot heterogeneity, and inter-zone handoffs give rise to a challenging variant of the flexible job shop problem with embedded transbot-routing features, resulting in substantial computational complexity. Motivated by a real manufacturing setting, two constraint programming formulations that integrate production scheduling with transbot routing are proposed: an arc-based formulation that explicitly models machine-to-machine transfers, and an operation-embedded formulation that embeds transfer decisions directly within the operation scheduling structure, leading to tighter synchronization between production and transportation decisions. Both formulations capture machine flexibility, zoning restrictions, handoff coordination, and collision-free path planning. To efficiently solve the resulting problem, a book-and-release strategy is proposed. It coordinates transbot movements without enforcing rigid routing patterns. Computational experiments on case study-adapted benchmark-based demonstrate that the proposed formulations generate high-quality solutions with strong computational performance. The generation of instances is described in detail to support future work in this emerging area.

\vspace{0.2cm}
{\footnotesize \emph{Keywords}: Manufacturing, Flexible Job Shop, Heterogeneous Transbots, Constraint Programming, Industry 5.0}\par
\vspace{0.2cm}

\end{abstract}

\newpage
\setlength{\parindent}{1em}
\setlength{\parskip}{0.5em}
\doublespacing
\newpage

\vspace{-3cm}

\section{Introduction}\label{section:1}


Growing demand for rapid fulfillment has increased pressure on manufacturers to improve responsiveness while maintaining cost efficiency. Advances in automation have encouraged firms to adopt technologies that mitigate labor shortages, address demand variability, and improve throughput. Large retailers and logistics providers, including Amazon, Walmart, and UPS, have incorporated AI-enabled transbots in their fulfillment operations to manage storage and material movement \citep{Repko2023WalmartAutomation, Dresser2025AmazonRobotics}. Such systems have been shown to enhance productivity, optimize space utilization, and improve workplace safety \citep{Jenkins2025WarehouseRobotics}. These developments motivate an examination of how transbot technologies can be integrated into manufacturing processes and how their inclusion affects overall workflow performance.
 
Rapid order fulfillment typically involves two primary stages: production and distribution. The present study draws on insights from intelligent manufacturing \citep{kusiak2019intelligent}. In manufacturing, workflow optimization focuses on determining the sequence of operations across machines and on evaluating system performance using metrics such as makespan, defined as the time required to complete all jobs. Each job consists of a sequence of operations needed to produce a finished good, and the makespan is strongly influenced by the degree of machine flexibility. When each operation can be processed on exactly one machine, the resulting classical job shop scheduling problem has been extensively studied \citep{xiong2022survey}.
In contrast, systems in which operations can be processed on multiple machines can exhibit partial or total flexibility \citep{dauzere2024flexible}. This setting gives rise to the flexible job shop problem (FJSP), which extends classical sequencing decisions by assigning each operation to an eligible machine out of a series of alternatives. Flexible job shop models capture important features of real manufacturing environments, including variability in resource availability and machine reliability. The models developed in this study are based on the flexible job-shop framework.

In many manufacturing facilities, machines are arranged non-sequentially across the shop floor. In such disjointed layouts, partially completed parts must be transported between machines between operations, and the time required for these transfers contributes directly to the makespan. To mitigate this effect, manufacturers increasingly adopt automated material-handling systems, paralleling the use of warehouse robotics in fulfillment operations. Common platforms include automated guided vehicles (AGVs) and autonomous mobile robots, which navigate using markers or sensors such as light detecting and ranging (LiDAR), as well as collaborative robots that operate alongside human workers \citep{Jenkins2025WarehouseRobotics}. In fulfillment centers, inventory is typically organized into designated areas, and AGVs follow predetermined paths to transport items while minimizing travel time and avoiding collisions. This problem can be seen as a variant of the capacitated vehicle routing problem, which is known to be NP-hard \citep{mo2024capacitated}. The present study integrates similar transbot systems into the flexible job shop environment to coordinate material transfers alongside production scheduling (see Figure \ref{fig:transbot}).

\begin{figure}[t!]
    \centering
    \includegraphics[width=0.9\textwidth]{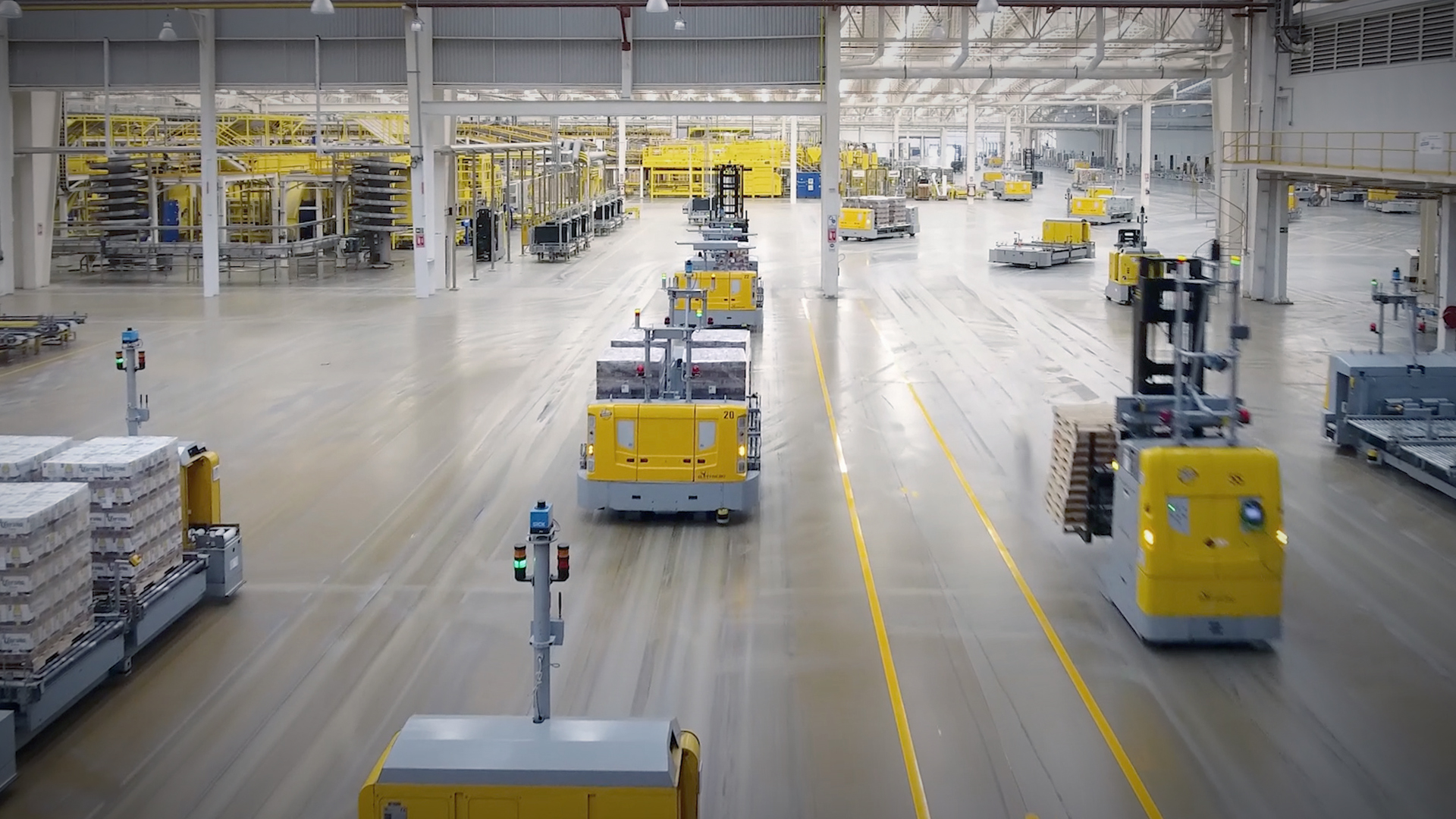}
    \caption{Transbots on a Manufacturing Floor}
    \label{fig:transbot}
\end{figure}

The flexible job shop problem with transbots (FJSPT) considers the simultaneous assignment of machines to operations and transbots to material transfers. The combination of these two NP-hard components results in a complex solution space that has received limited attention in the literature, particularly in the context of constraint programming (CP) formulations. Existing models, such as that of \cite{ham2020transfer}, assume homogeneous robots capable of serving all machines. More realistic manufacturing environments require additional considerations, including transbot-specific weight capacities and restrictions to designated areas corresponding to accessible machines. CP has been increasingly applied to scheduling, routing, and resource allocation problems due to its effectiveness in solving large-scale, combinatorial instances \citep{hooker2002logic}.

This study is motivated by a leading experimental manufacturing facility at Georgia Tech, United States. The facility provides machine and job workflows to external organizations seeking to develop new parts. The objective is to simulate production at maximum capacity by coordinating transbots and machines on the shop floor. A key requirement is to accommodate transbots within different zones. Transbots follow fixed paths and are restricted to designated zones, with intersections reserved for handoffs to ensure coordinated, collision-free material transfers. A \emph{zone} is a predefined subset of machines to which a transbot is exclusively assigned and within which it can operate freely. A \emph{handoff} point is a designated transfer location that enables parts to move between zones by allowing transbots from different zones to exchange loads. These characteristics give rise to the flexible job shop problem with heterogeneous transbots (FJSPT-H).

This paper proposes a threefold contribution as follows:

\begin{enumerate}
    \item \textbf{CP models for the FJSPT-H.} Two CP formulations for the FJSPT-H are proposed. The first is an \emph{arc-based} formulation, in which material transfers are modeled explicitly through arcs connecting pickup and drop-off machines. The second is an \emph{operation-embedded} formulation that integrates transfer decisions directly into the sequencing of operations. Both formulations extend the FJSPT beyond the standard assumption of identical transbots by explicitly considering heterogeneous transbots. To manage congestion in high-density facilities, the FJSPT-H incorporates zoning constraints that restrict each transbot to a subset of machines located within a given zone. These restrictions also reflect practical security and safety requirements, as certain transbots are authorized to operate only within specific zones. Designated handoff points allow transbots to exchange materials when consecutive operations belong to different zones. From a modeling perspective, a handoff point is treated similarly to a machine that performs no operations. Overall, zoning reduces congestion and enhances coordination by controlling access to sensitive or high-traffic areas by transbots.
   \item \textbf{A comparison between formulations.} A systematic comparison in terms of computational complexity, model size, and scalability is provided. The arc-based formulation offers a more compact representation for problems with many inter-zone transfers. In contrast, the operation-embedded formulation provides tighter integration of transfer scheduling and machine operations, potentially improving solution quality at the expense of a larger search space.
   \item \textbf{A book-and-release strategy for transbot routing.} The \emph{book-and-release} strategy reserves a path exclusively for a single transbot, preventing conflicts and ensuring collision-free movement. The model also incorporates acceleration strategies, including valid inequalities and parameter tuning for the CP solver. This path isolation, combined with acceleration strategies, enables more efficient exploration of the solution search space.
   \item \textbf{Extensive computational results.} The CP models for the FJSPT-H are tested on small- and medium-scale instances derived from standard benchmarks \citep{deroussi2010simultaneous} and adapted to reflect the experimental manufacturing facility at Georgia Tech. This paper addresses these instances by incorporating a transfer-time matrix for transbots, designated handoff points, and fixed paths corresponding to the original operation and machine sequences. Empirical results highlight the trade-offs between these approaches, guiding practitioners in selecting the appropriate formulation depending on problem characteristics and computational resources. The generation of instances is described in detail to support future work in this emerging area and facilitate comparison with alternative approaches.
\end{enumerate}

The remainder of this paper is structured as follows. Section \ref{section:2} discusses the literature review. Section \ref{section:3} describes the problem. Sections \ref{section:4} and \ref{section:5} describe the two CP formulations. Section \ref{section:6} compares the models and discusses the solution methodology. Section \ref{section:7} discusses the computational results, while Section \ref{section:8} summarizes the conclusions. 

\section{Literature Review}\label{section:2}

The FJSP has been evolving significantly over the past few years (for recent surveys, see \cite{zhang2019review} and \cite{destouet2023flexible}). This section reviews the FJSPT variants. Then, it discusses the optimization techniques used to solve it.

Compared with the FJSP, the FJSPT is more recent and was introduced with the advent of transbots in manufacturing. In the FJSPT context, the makespan incorporates transfer time between sequential operations. Most studies evaluate systems with homogeneous transbots, which facilitates flexibility well, but can be unrealistic in practical manufacturing settings. \cite{deroussi2010simultaneous} propose the first small-scale benchmark instances for the FJSPT, extending the original \cite{bilge1995time} job-shop instances, in which machine sequences are fixed for each job, by introducing machine flexibility that allows each operation to be performed on up to two alternative machines across 10 instances. This small-scale setting considers instances with at most 21 operations. \cite{kumar2011simultaneous} curate a set of 7 job sets with 4 layout matrices based on \cite{bilge1995time} cases, altering the ratio of travel time between machines to processing time of an operation on a machine. \cite{ham2020transfer} adopt the medium-scale instance set from \cite{hurink1994tabu} and incorporate 3 varied transition time matrices. This set encompasses test cases with up to 300 operations. \cite{homayouni2023multistart} consists of medium-scale instances from \cite{brandimarte1993routing} and \cite{chambers1996new} with 55 to 240 operations. This paper proposes extensions to the instances of \cite{deroussi2010simultaneous} and \cite{ham2020transfer} by incorporating handoff points and zoning into the transfer time matrices for transbots.

In isolation, the FJSP is NP-hard \citep{dauzere2024flexible}. The scheduling of transbots, which can be seen as a vehicle routing problem, is itself NP-complete \citep{lenstra1981complexity}. Consequently, integrating job sequencing with transbot scheduling results in a highly complex optimization problem that requires advanced solution strategies. Two complementary surveys by \cite{berterottiere2024flexible} and \cite{xin2025review} review popular approaches used to tackle the FJSPT. 

Initial studies first incorporated transportation into the job shop problem (JSP), in which each job follows a fixed machine sequence, offering no machine flexibility. \cite{raman1986simultaneous} introduces a single shared transportation resource to a shop floor. \cite{bilge1995time} explore the addition of a small number of identical transportation resources by formulating a mixed integer programming (MIP) model and solving it using an iterative heuristic procedure, referred to as the sliding time window, which alternates between generating machine schedules and constructing feasible transportation schedules. \cite{hurink2005tabu} analyze JSP with a single transportation resource by modeling transportation as an additional operation and applying tabu search in both a one-stage simultaneous scheduling approach and a two-stage approach in which robot scheduling follows machine scheduling.

When machine flexibility is incorporated alongside operation sequencing and vehicle assignment, the problem becomes the FJSP. In addition to introducing the first benchmark instances for FJSPT, \cite{deroussi2010simultaneous} propose an iterated local search metaheuristic that uses a discrete-event simulation model to evaluate the makespan. \cite{deroussi2014hybrid} introduces a hybrid metaheuristic that integrates particle swarm optimization with stochastic local search for the FJSPT. \cite{nouri2016simultaneous} develop a hybrid metaheuristic that integrates a neighborhood-based genetic algorithm with iterative local search to handle simultaneous scheduling with multiple vehicles. \cite{karimi2017scheduling} model transportation as fixed lags depending on operation-machine assignment. They proposed sequence-based and position-based MIP models, assuming an infinite number of transportation resources to eliminate any transportation-induced delays. \cite{yan2021research} address the simultaneous scheduling of the FJSP with transportation conditions using an improved genetic algorithm, encoding layers for operation sequencing, machine assignment, and transportation decisions. \cite{homayouni2021production} propose a MIP to solve small and medium-scale FJSPT instances to optimality. They develop a late acceptance hill-climbing heuristic to address larger instances. \cite{homayouni2023multistart} further improve performance on large-scale instances by introducing a multistart biased random key genetic algorithm. \cite{yao2024novel} extend the problem to a multi-objective setting that considers both the makespan and total energy consumption of transportation resources, and propose a knowledge-based evolutionary algorithm for joint scheduling of machines and vehicles.  
\cite{berterottiere2024flexible} design a tabu search metaheuristic with a large neighborhood move structure and a constant time move evaluation procedure to efficiently explore machine and vehicle scheduling within an extended disjunctive graph model. More recently, \cite{he2025hybrid} apply ant colony optimization, while \cite{yao2026knowledge} introduce a knowledge-enhanced discrete artificial bee colony algorithm to tackle simultaneous scheduling in the FJSPT.

Scalability remains a fundamental challenge in FJSPT modeling, particularly as problem instances increase in size and complexity. CP approaches have risen in popularity due to their efficiency in solving large-scale scheduling, routing, and resource allocation problems. It emphasizes constraint propagation and search rather than linear relaxations. Modern CP solvers integrate automated search procedures, providing optimal integer solutions or proving infeasibility otherwise \citep{naderi2023mixed}. Additionally, CP approaches outperform MIPs on large-scale FJSP instances \citep{ku2016mixed} by drastically reducing the search space via precedence constraints, thereby progressively narrowing the set of feasible solutions.

\cite{ham2020transfer} and \cite{berterottiere2024flexible} are the most pertinent papers to our research regarding a CP solution to FJSPT. \cite{ham2020transfer} proves optimality and supersedes exact methods for \cite{deroussi2010simultaneous} instances that were previously only heuristically solved using two distinct approaches. One approach designs separate pickup and dropoff interval variables to handle material transport, but this is outperformed by the second, which merges the pickup and dropoff variables into a single transfer variable handled by a single transbot. Comparable literature includes \cite{booth2016constraint}, which presents an approach to multi-robot task allocation and scheduling in retirement homes, and implements symmetry breaking, variable ordering heuristics, and large neighborhood search. Similarly, \cite{11339952} use a CP approach to enable operations that do not require transportation and is the first model to prove the optimality on all \cite{kumar2011simultaneous} instances.
 
Despite the growing body of work on flexible job shop scheduling with transportation resources, existing models typically assume identical transfer vehicles, unrestricted mobility, or simplified routing structures. In contrast, this paper introduces two CP formulations for the FJSPT that explicitly capture heterogeneous transbots, zoning restrictions, and handoff points within a single scheduling–routing framework. Beyond modeling expressiveness, the proposed formulations leverage global constraints to achieve substantial computational gains. This combination of efficiency and scalable solution techniques distinguishes the proposed FJSPT-H from prior approaches.

\section{Problem Description}\label{section:3}

Consider a shop floor comprised of a set of machines $\emph{m} \in \mathcal{M}$ that are non-sequentially laid out on the floor, and a set of transbots $\emph{v} \in \mathcal{V}$ responsible for transferring materials between machines during sequential operations (see Figure \ref{fig:layout}). There is a set of jobs $\mathcal{J}$ and operations $\mathcal{O}$, where each operation $o \in \mathcal{O}$ belongs to exactly one job $j_o \in \mathcal{J}$. The shop floor is partitioned into a set of zones $z \in \mathcal{Z}$. Each machine $m \in \mathcal{M}$ and each transbot $v \in \mathcal{V}$ is assigned to a unique zone, denoted by $z_m$ and $z_v$, respectively.

\begin{figure}[t!]
    \centering
    \includegraphics[width=0.85\textwidth]{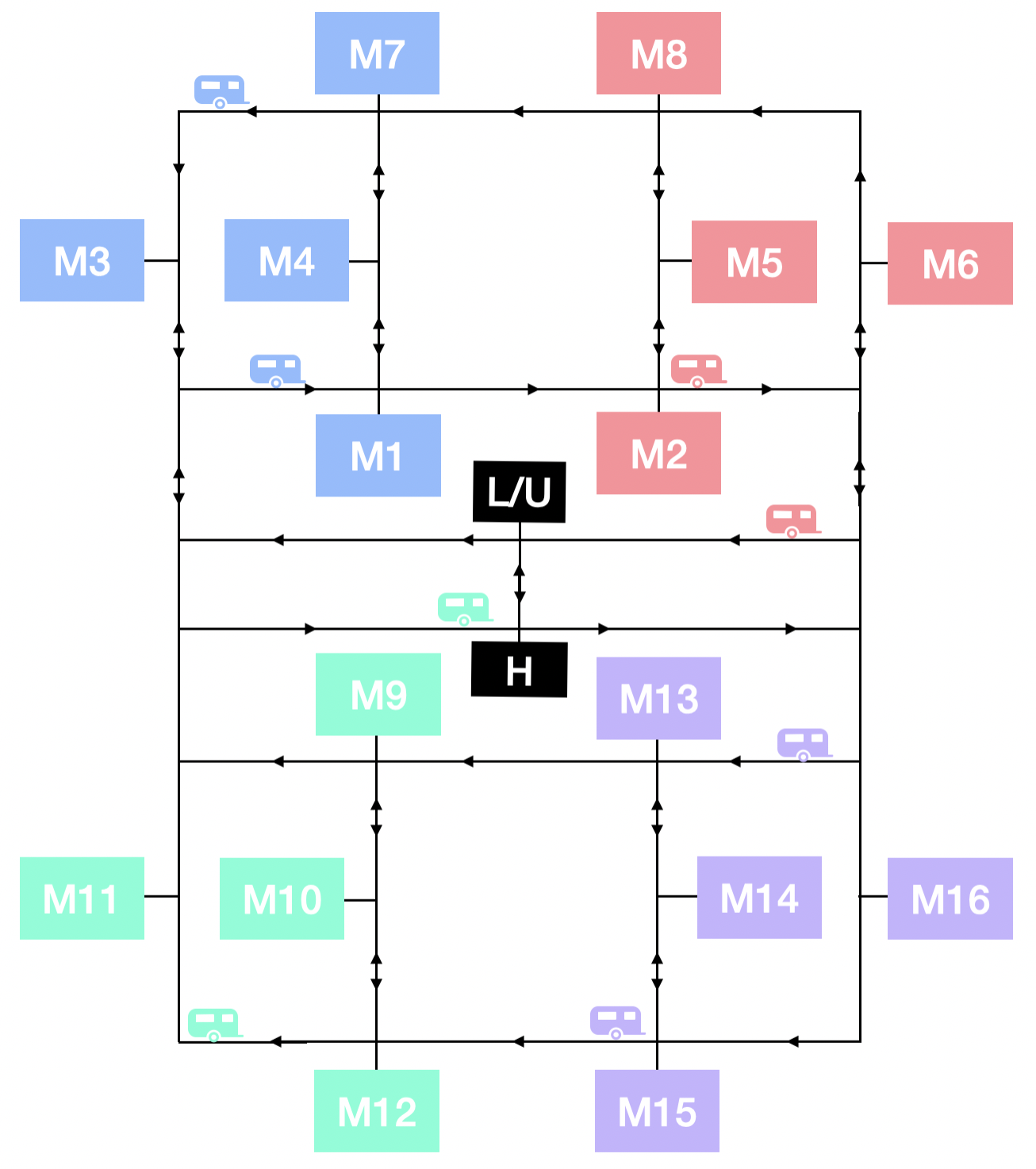}
    \caption{Layout of a shop floor with a stocker, a handoff point, and four zones, each with four machines and two transbots.}
    \label{fig:layout}
\end{figure}

Each operation $o \in \mathcal{O}$ is assigned an order index $g_o \in \mathbb{N}$ such that, for any two operations $o$ and $o'$ belonging to the same job, $g_o < g_{o'}$ if and only if $o$ must be completed before $o'$. An operation can be processed on one machine from a set of alternatives $\mathcal{M}_{o} \subseteq \mathcal{M}$. However, a machine can only execute one operation at a time. Depending on each pair of operation $o \in \mathcal{O}$ and machine $m \in \mathcal{M}$, there is a corresponding processing time $p_{om}\in\mathbb{R}$. All raw materials reside at the load/unload station, referred to as the stocker and denoted as $\emph{L/U}$. They must be transported to the machine before the first operation is executed for each job. Transbots are not required to return finished goods to the stocker after the final operation of a job; instead, the job is assumed to be completed at the machine where its last operation is processed.

\begin{table}[t!]
\centering
\begin{tabular}{p{0.2\textwidth}|p{0.75\textwidth}}
\toprule
 \textbf{Notation} & \textbf{Description} \\ 
 \midrule
 $m \in \mathcal{M}$ & machines \\
 \midrule
 $v \in \mathcal{V}$ & transbots \\
 \midrule
 $o \in \mathcal{O}$ & operations \\
 \midrule
 $j \in \mathcal{J}$ & jobs where $j_o$ is the job of operation $o$\\
 \midrule
 $z \in \mathcal{Z}$& zones\\
 \midrule
 $g_o \in \mathbb{N}$ & order index of operation $o$ in job $j_o$\\
 \midrule
 $M_o \subseteq M$ & set of machines eligible to process operation $o$ \\
 \midrule
 $p_{om} \in \mathbb{R}$ & processing time of operation $o$ on machine $m$\\
 \midrule
 $a \in \mathcal{A}$ &  set of all feasible arcs linking two machines $m$ and $m'$\\
 \midrule
 $l\in \mathcal{L}_{a}$ & legs composing arc $a \in \mathcal{A}$\\
 \midrule
 $k_l \in \{1,...,|\mathcal{L}_a|\}$ & index position of leg $l$ within arc $a$\\
 \midrule
 $i_l\in \mathbb{N}$ & unique identifier of leg $l$\\
 \midrule
 $T$ & travel times matrix \\
 \midrule
 $t \in \mathbb{R}$ & travel time with $t_a$ for arc $a$ and $t_l$ for leg $l$\\
 \midrule
 $L/U$ & stocker \\
 \midrule
 $H$ & handoff point\\
 \bottomrule
\end{tabular}
\caption{Problem Notation}
\label{notation}
\end{table}

Every operation induces a transfer task that is associated with a possible directed arc $a \in \mathcal{A}$, potentially composed of multiple legs, which abstracts the underlying path followed by a transbot between two machines $m$ and $m'$, where \(m\) and \(m'\) denote the pickup and drop-off machines, respectively. When the context requires it, pickup and dropoff machines of arc $a$ are denoted $m_a$ and $m'_a$. Travel times between machines are given in a matrix \(T\). Each arc $a \in \mathcal{A}$ is decomposed into an ordered set of legs $\mathcal{L}_a = \{ l_k : k = 1, \dots, |\mathcal{L}_a| \}$, where the index $k$ specifies the traversal order of the legs along arc $a$. If both machines belong to the same zone, then there is only one leg ($|\mathcal{L}_a| = 1$). This transfer will be handled by one of the transbots that matches the zone $z_m = z_v$. However, should the machines of the assigned arc belong to different zones, then two legs are required to complete the singular transfer, i.e., $\mathcal{L}_{a} = \{l_1, l_2\}$, and each leg is handled by its respective transbots, matched by zone. Denote by $z_l$ the zone corresponding to leg $l \in \mathcal{L}_a, \ a \in \mathcal{A}$. In this case, the first transbot must pick up the material from the initial station $m$ and travel to the handoff point $H$. The second transbot will retrieve the material from $H$ and deliver it to the machine $m'$, performing the next operation. The matrix $T$ includes both $H$ and $L/U$ as stations that do not perform operations. Each leg is also labeled with a unique index $i_{l} \in \mathcal{I}, \ l \in\mathcal{L}_{a}, \ a \in \mathcal{A}$. This notation inherently maps each leg to its parent arc. If two consecutive operations are processed on the same machine, i.e., \(m = m'\), no transportation is required, and the corresponding transfer task is empty, yielding \(|\mathcal{L}_a| = 0\). The problem notation is summarized in Table \ref{notation}.

This paper adheres to the same assumptions made regarding the FJSPT in \cite{bilge1995time}: (i) Machine operations and transbot transfers are non-preemptive and there is sufficient input and output (I/O) buffer space at each machine and $L/U$ stocker to avoid deadlocks, (ii) Processing, loading and unloading times are factored into all time-related variables, (iii) The number of transbots is known and they initially start from the stocker, and (iv) Transbots have unit-capacity and they move along predetermined shortest paths, with the assumption of no delay due to congestion and travel times on each segment of the path are known. 

\begin{table}[t!]
\centering
\begin{tabular}{p{0.1\textwidth}|p{0.85\textwidth}}
\toprule
 \textbf{Variable} & \textbf{Description} \\ 
 \midrule
 $x_o$ & transfer interval variable for operation $o$ \\
 \midrule
 $x_{oa}$ & optional interval variable representing the chosen arc $a$ for the transfer of operation $o$ \\
 \midrule
 $x_{olv}$ &  optional interval variable assigning transbot $v$ to each leg $l$ of arc $a$ for the transfer or operation $o$\\
 \midrule
 $y_o$ & operation interval variable of operation $o$\\
 \midrule
 $y_{m}$& optional interval variable representing the selection of machine $m$ for operation $o$\\
 \midrule
 $w_m$& sequence variable storing the ordering of operations $o$ on each machine $m$\\
 \midrule
 $w_{v}$ & sequence variable to track each machine a transbot $v$ visits by storing the unique identifier $i_{l}$ of each leg to access the matrix entry of $T$\\
 \bottomrule
\end{tabular}
\caption{Arc-based Formulation Decision Variables}
\label{arc-based notation}
\end{table}


\vspace{0.5cm}

\section{Arc-based Formulation}\label{section:4}

The FJSPT-H is formulated using CP, with the relevant CP functions and constructs described in Appendix \ref{A}. In addition to the indices, sets, and parameters presented in Section \ref{section:3}, Table \ref{arc-based notation} summarizes the decision variables. The durations of $x_o$ and $x_{oa}$ are the sum of the travel times on all the legs $l \in \mathcal{L}_a$ forming arc $a$, where $t_a = \sum_{l \in L_{a}} t_{l}$. The length of $x_{olv}$ is $t_l$. The length of $y_o$ and $y_{m}$ is the processing time of operation $o$ on machine $m$, i.e., $p_{om}$. The arc-based formulation is the following:

\begingroup
\allowdisplaybreaks
\begin{flalign}
& \min \; \max_{o \in \mathcal{O}}\; \mathtt{endOf}(y_{o})
\label{Objective}
&
\\[0.6em]
&
\mathtt{alternative}\!\left(
    y_o, \{y_{m}\}_{m \in \mathcal{M}_o} \right), \hspace{1mm} \forall o \in \mathcal{O}
\label{c1 - alternative(job,jobOnMach)}
&
\\[0.6em]
&
\mathtt{alternative}\!\left( x_o, \{x_{oa}\}_{a \in \mathcal{A}
}\right), \hspace{1mm} \forall o \in \mathcal{O}
\label{c2 - alternative(trans,transArc)}
&
\\[0.6em]
&
\mathtt{span}\!\left(
    x_{oa},\{x_{o l v}\}_{l \in \mathcal{L}_a,v \in \mathcal{V}}
\right),
\hspace{1mm} \forall o \in \mathcal{O},\ a\in \mathcal{A}
\label{c13 - sync duration}
&
\\[0.6em]
&
\mathtt{endBeforeStart}(x_o,\, y_o), \hspace{1mm} \forall o \in \mathcal{O}
\label{c3 - endBeforeStart(trans,job)}
&
\\[0.6em]
&
\mathtt{endBeforeStart}(y_o,\, x_{o'}), \hspace{1mm} \forall o,o' \in \mathcal{O}:\ j_o = j_{o'}, g_{o'} = g_o + 1
\label{c4 - endBeforeStart(job,transfer)}
&
\\[0.6em]
&
\mathtt{presenceOf}(x_{oa}) \hspace{-2mm} \implies \hspace{-2mm} \mathtt{presenceOf}(y_{m}), \hspace{0.5mm} \forall o \in \mathcal{O}, m \in \mathcal{M}_o, \ a \in \mathcal{A}:\ g_o = 1,\ m_a = L/U,\ m'_a=m 
\label{c5 - start at stocker}
&
\\[0.6em]
&
\mathtt{presenceOf}(x_{oa}) \hspace{-1mm} \implies \hspace{-1mm} \mathtt{presenceOf}(y_{m}), \hspace{0.5mm} \forall o \in \mathcal{O},\ m \in \mathcal{M}_o, a\in \mathcal{A}: g_o \neq 1, \ m'_a = m 
\label{c6 - link station 2}
&
\\[0.6em]
&
\mathtt{presenceOf}(x_{oa}) \hspace{-2mm} \implies \hspace{-2mm} \mathtt{presenceOf}(y_{m}), \hspace{0.5mm} \forall o',o \in \mathcal{O},\ m \in \mathcal{M}_{o'}, \ a \in \mathcal{A}: g_o \neq 1, g_{o'}=g_o-1, m_a=m 
\label{c7 - link station 1}
&
\\[0.6em]
&
\mathtt{presenceOf}(x_{oa})= \sum_{v \in \mathcal{V}: z_v = z_l} \mathtt{presenceOf}(x_{olv}), \hspace{1mm} \forall o \in \mathcal{O},\ l \in \mathcal{L}_a,\ a \in \mathcal{A}\ 
\label{c10 - assign vehicle}
&
\\[0.6em]
&
\sum_{v \in \mathcal{V}} \mathtt{presenceOf}(x_{o l v}) \le 1,
\forall o \in \mathcal{O},\ 
l \in \mathcal{L}_a,\ 
a \in \mathcal{A}\label{c11 - limit one vehicle to each arc}
&
\\[0.6em]
&
\mathtt{presenceOf}(x_{o l' v}) \hspace{-1mm} \implies \hspace{-1mm}
\mathtt{startOf}(x_{o l' v}) \ge \mathtt{endOf}(x_{olv}),
\forall o \in \mathcal{O},\ 
l \in \mathcal{L}_a,\ 
a \in \mathcal{A},\ 
v \in \mathcal{V}:\  \hspace{-2mm} 
k_{l'} = k_l +1 \label{c12 - sort precedence of legs}
&
\\[0.6em]
&
\mathtt{noOverlap}(w_v, T, \mathrm{false}), \hspace{1mm} \forall v \in \mathcal{V}
\label{c8 - noOverlap(vehicles)}
&
\\[0.6em]
&
\mathtt{noOverlap}(w_m), \hspace{1mm} \forall m \in \mathcal{M}
\label{c9 - oOverlap(machines)}
&
\end{flalign}
\endgroup

\noindent Objective \eqref{Objective} minimizes the makespan, i.e., total time required to complete all jobs, including the travel time for transbots transferring material between machines. Global \texttt{Alternative} Constraints \eqref{c1 - alternative(job,jobOnMach)} enforce that each operation interval in $y_o$ is assigned to exactly one machine interval from $y_{m}$ while synchronizing their start and end times. Similarly, Constraints \eqref{c2 - alternative(trans,transArc)} assign an appropriate pickup-dropoff machine arc to each transfer within $x_o$, synchronizing its time-span with the travel time between the machines of this activated arc in $x_{oa}$. The Global \texttt{span} Constraints \eqref{c13 - sync duration} synchronizes the time span of the transfer arc interval $x_{oa}$ with the duration of its corresponding legs $x_{olv}$. Before the start of the first operation in a job, the transbot must first retrieve the material from the stocker to deliver it to the first machine.  Therefore, transfers $x_o$ are always ordered before the execution of operations $y_o$ as delineated with the global \texttt{endBeforeStart} Constraints \eqref{c3 - endBeforeStart(trans,job)}. These constraints are reinforced by Constraints \eqref{c4 - endBeforeStart(job,transfer)}, which declare that the previous operation $y_o$ in a job must be completed before the transfer for the next operation $x_{o'}$ can begin. Additionally, Constraints \eqref{c5 - start at stocker} enforce that the initial transfers of each job begin at the stocker. Constraints \eqref{c6 - link station 2} and \eqref{c7 - link station 1} ensure that if a transfer along arc $a$ is selected for operation $o$, then operation $o$ must be executed on machine $m'_a$. The previous operation $o'$ should have been processed on machine $m_a$. Constraints \eqref{c10 - assign vehicle} ensure that a transbot is assigned to a specific leg of a transfer only if that transfer is active. The assignment respects zone restrictions, meaning the transbot must operate within the same zone as the leg. Constraints \eqref{c11 - limit one vehicle to each arc} prevent the assignment of redundant transbots to each leg. The precedence of the legs is defined in Constraints \eqref{c12 - sort precedence of legs}. Global \texttt{noOverlap} Constraints \eqref{c8 - noOverlap(vehicles)} and \eqref{c9 - oOverlap(machines)} prevent the intervals contained within the sequence variables from overlapping.

\vspace{0.5cm}

\begin{table}[H]
\centering
\begin{tabular}{p{0.1\textwidth}|p{0.85\textwidth}}
\toprule
\textbf{Variable} & \textbf{Description} \\ 
\midrule
 $x_o$ & transfer interval variable for operation $o$ \\
\midrule
 $x_{lv}$ & optional interval variable assigning a transbot $v$ to each leg $l \in \mathcal{L}_{a}$\\
 \hline
 $y_o$ & operation interval variable for processing of operation $o$\\
\midrule
 $y_{a}$& optional machine-pair interval variables embedding the processing of operation $o$ on the dropoff machine $m'_a$ and the pick-up machine $m_a$ on which the preceding operation $o'$ was processed\\
\midrule
 $w_m$& sequence variable storing the ordering of operations $o \in \mathcal{O}$ on machine $m$\\
\midrule
 $w_{v}$ & sequence variable tracking each machine a transbot $v$ visits by storing the unique index $i_{l}$ of each leg $l$ to access the matrix entry of $T$\\
\bottomrule
\end{tabular}
\caption{Operation-embedded Formulation Decision Variables}
\label{operation-embedded notation}
\end{table}

\newpage

\section{Operation-embedded Formulation} \label{section:5}

In addition to the indices, sets, and parameters presented in Section \ref{section:3}, Table \ref{operation-embedded notation} summarizes the decision variables. The operation-embedded formulation is the following:

\begingroup
\allowdisplaybreaks
\begin{flalign}
& \min \; \max_{o \in \mathcal{O}}\; \mathtt{endOf}(y_{o})
\label{m2Objective}
&
\\[0.6em]
&
\mathtt{alternative}\!\left(
    y_o, y_{a}\right), \hspace{1mm} \forall o \in \mathcal{O}, a\in A: \ m'_a \in {M}_{o}
\label{m2c1 - alternative(job,newJobOnMac)}
&
\\[0.6em]
&
\mathtt{span}\!\left(
    x_o,{x_{lv}}
\right),
\hspace{1mm} 
\forall o \in \mathcal{O},
\ l \in {L}_{a},
\ a \in \mathcal{A},
\ v \in V: m'_a \in \mathcal{M}_o
\label{m2c2 - sync duration}
&
\\[0.6em]
&
\mathtt{endBeforeStart}(x_o,\, y_o), \hspace{1mm} \forall o \in \mathcal{O}
\label{m2c4 - endBeforeStart(trans,job)}
&
\\[0.6em]
&
\mathtt{endBeforeStart}(y_o,\, x_{o'}), \hspace{1mm} \forall o,o' \in \mathcal{O}:\ j_o = j_{o'}, \ g_{o'} = g_o + 1
\label{m2c5 - endBeforeStart(job,transfer)}
&
\\[0.6em]
&
\mathtt{presenceOf}(y_{o'a}) = 
\sum_{a \in \mathcal{A}} \mathtt{presenceOf}(y_{a}),\hspace{1mm}
\forall o,o' \in \mathcal{O},\ 
a \in \mathcal{A}: 
j_o = j_{o'},
g_{o} = g_{o'} - 1, 
\ g_{o'} \ne 1, \ 
m'_a \in \mathcal{M}_{o'} 
\label{m2c6 - station linking}
&
\\[0.6em]
&
\mathtt{presenceOf}(y_{a}) =
\sum_{v \in \mathcal{V}: z_v = z_l} \mathtt{presenceOf}(x_{lv}), \ 
\forall o \in \mathcal{O},\ 
l \in {L}_{a}, \ 
a \in \mathcal{A}: m'_a \in \mathcal{M}_o
\label{m2c7 - assign vehicle}
&
\\[0.6em]
&
\mathtt{presenceOf}(y_{a}) =
\sum_{v \in \mathcal{V}} \mathtt{presenceOf}(x_{lv})\ 
\forall o \in \mathcal{O},\ 
l \in {L}_{a}, \ 
a \in \mathcal{A}: m'_a \in \mathcal{M}_o
\label{m2c8 - restrict other vehicles}
&
\\[0.6em]
&
\mathtt{presenceOf}(x_{l'v'}) \hspace{-2mm} \implies \hspace{-2mm}
\mathtt{startOf}(x_{l'v'}) \ge \mathtt{endOf}(x_{lv}), \hspace{1mm} 
\forall o \in \mathcal{O},\ 
l \in \mathcal{L}_{a},\ 
a \in \mathcal{A},\ 
v,v' \in \mathcal{V}:\ 
m'_a \in \mathcal{M}_o,\
k_{l'} = k_l+1 
\label{M2c9 - sort precedence of legs}
&
\\[0.6em]
&
\mathtt{noOverlap}(w_v, T, \mathrm{false}), \hspace{1mm} \forall v \in \mathcal{V}
\label{m2c10 - noOverlap(vehicles)}
&
\\[0.6em]
&
\mathtt{noOverlap}(w_m), \hspace{1mm} \forall m \in \mathcal{M}
\label{m2c11 - noOverlap(machines)}
&
\end{flalign}
\endgroup

\noindent Objective \eqref{m2Objective} minimizes the makespan. Constraints \eqref{m2c1 - alternative(job,newJobOnMac)} enforce that operation $o$ is executed on its selected drop-off machine $m'_a$ by activating a feasible machine arc pair $y_{a}$, $\forall a \in \mathcal{A}_o$. The duration of $y_o$ is the corresponding processing time of operation $o$ on machine $m=m'_a$, i.e., $p_{om}$. Constraints \eqref{m2c2 - sync duration} synchronize the transfer interval $x_o$ with its associated legs $x_{lv}$, ensuring that $x_o$ starts with the first active leg and ends upon completion of the last (single-leg case included). The sequence of transfers and operation processing is enforced by Constraints \eqref{m2c4 - endBeforeStart(trans,job)} and \eqref{m2c5 - endBeforeStart(job,transfer)}. To ensure proper station linking for operations within the same job, Constraints \eqref{m2c6 - station linking} ensure that the pickup machine of operation $o$ matches the drop-off machine that processed the preceding operation $o'$. The first operation of each job ($g_o = 1$) is exempt because the input arc data specifies the stocker as the pickup machine for the first operation. Constraints \eqref{m2c7 - assign vehicle} and \eqref{m2c8 - restrict other vehicles} work in conjunction to ensure that each leg of an arc is assigned exactly one zone-compatible transbot. When multiple legs are present, their precedence is enforced by Constraints \eqref{M2c9 - sort precedence of legs}. Constraints \eqref{m2c10 - noOverlap(vehicles)} and \eqref{m2c11 - noOverlap(machines)} prevent the intervals contained within the sequence variables from overlapping.

\section{Comparison and Methodology} \label{section:6}

This section compares the two formulations and discusses the solution methodology.

\subsection{Formulation Comparison}\label{section:comparison}

\noindent Both the arc-based and operation-embedded formulations fully capture the integrated scheduling and routing requirements of the FJSPT-H. While the two formulations are equivalent in terms of feasible schedules and optimal makespan, they differ substantially in how decisions are represented and propagated. In particular, the operation-embedded formulation reorganizes routing and processing decisions into a tighter variable structure, leading to a more compact and propagation-friendly representation in practice.

\noindent \textbf{Visual Comparison.} Figure \ref{fig:visual_comparison} visualizes the differences in propagation and interval linking from a modeling perspective. For the arc-based formulation, transfers are first chosen at the arc level ($x_{oa}$), which then determines the execution machine ($y_{m}$), and induces the leg-robot assignments ($x_{olv}$), as shown in Figure \ref{fig:cp3_sched}. Many of these optional intervals correspond to structurally infeasible machine–arc combinations and must be eliminated dynamically through propagation.

\begin{figure}[t!]
 \begin{subfigure}{\textwidth}
     \centering
     \includegraphics[width=0.8\textwidth]{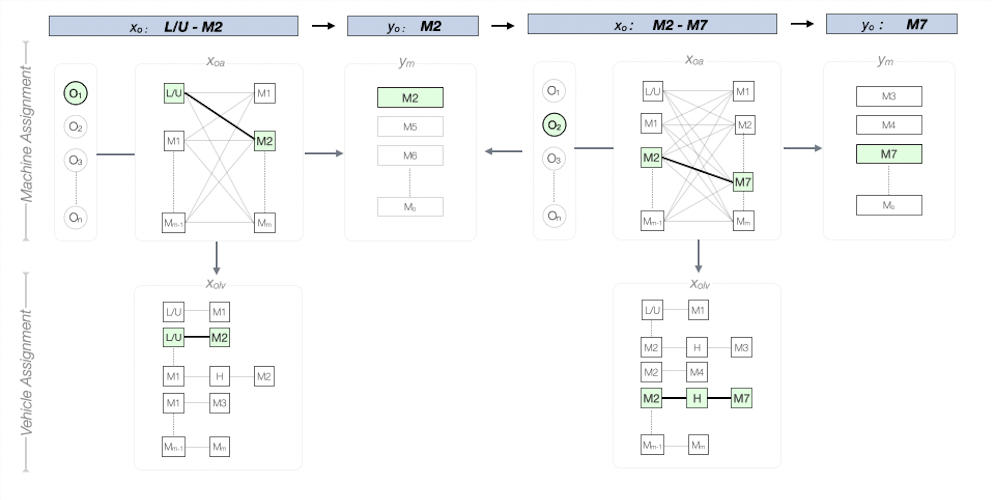}
     \caption{Arc-based Formulation}
     \label{fig:cp3_sched}
 \end{subfigure}
 \hfill
  \begin{subfigure}{\textwidth}
     \centering
     \includegraphics[width=0.8\textwidth]{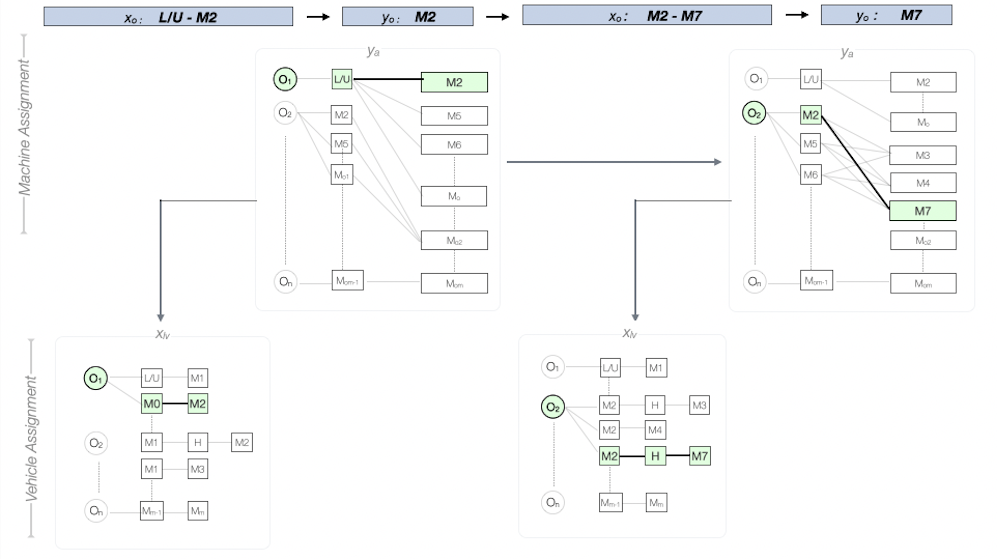}
     \caption{Operation-embedded Formulation}
     \label{fig:cp4_sched}
 \end{subfigure}
\caption{Visual comparison between the two formulations.}
\label{fig:visual_comparison}
\end{figure}

The operation-embedded formulation substantially improves propagation by introducing operation-embedded arcs and legs, as shown in Figure \ref{fig:cp4_sched}. Through aggressive preprocessing, the domain excludes impossible combinations \textit{a priori} by only mapping each operation to its viable arcs ($y_{a}$) and their associated legs ($x_{lv}: l \in \mathcal{L}_{a}$), reducing complexity in the decision space. Viable arcs for operation $o$ are defined as the set of all machine pairs formed by combining the feasible processing machines of the preceding operation $o'$ (pick-up machine $m_a\in M_{o'}$) with the series of processing machines of the current operation $o$ (dropoff machine $m'_a \in M_{o}$). The consecutive operations must belong to the same job ($j_o = j_{o'}$). This shift not only decreases the number of optional intervals but also ensures the remaining variables are prevalidated and feasible modeling choices. 

The modeling choice of decision variables in the operation-embedded formulation enables tighter and more effective constraint formulations. By leveraging preprocessed operation-embedded arcs, $y_{a}$ simultaneously governs machine assignment within the transfer selection without overconstraining the model, where the pickup machine $m$ initiates the transfer and dropoff machine $m'$ executes the operation. As a result, the operation-embedded formulation substantially reduces reliance on implication-based constraints used in the arc-based formulation for station linking. This can be understood by comparing Constraints \eqref{m2c6 - station linking} to Constraints \eqref{c6 - link station 2} and \eqref{c7 - link station 1}. This avoids conditional reasoning over a large search space, known to weaken propagation. Additionally, this preprocessing establishes more informative domains by incorporating static conditions. For example, the first operation for each job is always restricted to being picked up from the stocker, enabling transbots to serve subsets of machines while adhering to the need to encode this condition explicitly in the constraints, as is done in the arc-based formulation (see Constraints \eqref{c5 - start at stocker}). 

The arc-based formulation accommodates real-world manufacturing settings, such as incompatibilities arising from different transbot and machine vendors or facility layouts, by establishing zones and enabling handoff transfers between them. Transbot feasibility is therefore modeled at the machine level, allowing transbots to serve subsets of machines while respecting cross-zone transfer logic. Both formulations support machine- and zone-based transbot feasibility. However, the operation-embedded formulation incorporates this heterogeneity directly into operation-level arc variables, allowing routing feasibility to be filtered during preprocessing rather than via implication constraints. In a manufacturing context, this enables transbots to be customized to varying load capacities or handling requirements. This granularity provides a more compact, propagation-oriented representation of shop-floor operations without compromising modeling expressiveness.

\begin{table}[t!]
  \centering
  \caption{Comparison of Model Decision Variables}
    \begin{tabular}{p{0.12\textwidth}|>{\centering\arraybackslash}p{0.13\textwidth}|>{\centering\arraybackslash}p{0.23\textwidth}|>{\centering\arraybackslash}p{0.13\textwidth}|>{\centering\arraybackslash}p{0.23\textwidth}}
    \toprule
    \multicolumn{1}{c|}{\multirow{2}{*}{\textbf{Type}}} 
    & \multicolumn{2}{c|}{\textbf{Arc-Based}} 
    & \multicolumn{2}{c}{\textbf{Operation-Embedded}} \\
\cmidrule{2-5}
    & \textbf{Variable} & \textbf{Domain} 
    & \textbf{Variable} & \textbf{Domain} \\
    \midrule
    
    \textit{transfer} 
    & $x_o$ 
    & $|\mathcal{O}|$ 
    & $x_o$ 
    & $|\mathcal{O}|$ \\
    
    \midrule
    
    \textit{transArc} 
    & $x_{oa}$ 
    & $\displaystyle \sum_{o \in \mathcal{O}} |\mathcal{A}_o|$ 
    & -- 
    & -- \\
    
    \midrule
    
    \textit{transOnVeh} 
    & $x_{olv}$ 
    & $\displaystyle \sum_{o \in \mathcal{O}} \sum_{l \in \mathcal{L}_o} |\mathcal{V}_{z_l}|$ 
    & $x_{lv}$ 
    & $\displaystyle \sum_{l \in \mathcal{L}} |\mathcal{V}_{z_l}|$ \\
    
    \midrule    
    
    \textit{job} 
    & $y_o$ 
    & $|\mathcal{O}|$ 
    & $y_o$ 
    & $|\mathcal{O}|$ \\
    
    \midrule
    
    \textit{jobOnMach} 
    & $y_{m}$ 
    & $\displaystyle \sum_{o \in \mathcal{O}} |\mathcal{M}_o|$ 
    & $y_{a}$ 
    & $\displaystyle \sum_{o \in \mathcal{O}} |\mathcal{A}_o|$ \\
    
    \midrule
    
    \textit{seqMach} 
    & $w_m$ 
    & $|\mathcal{M}|$ 
    & $w_m$ 
    & $|\mathcal{M}|$ \\
    
    \midrule
    
    \textit{seqVeh} 
    & $w_v$ 
    & $|\mathcal{V}|$ 
    & $w_v$ 
    & $|\mathcal{V}|$ \\
    
    \bottomrule
    \end{tabular}
  \label{tab:comparison_variables}
\end{table}

\noindent \textbf{Variable Comparison.} Table \ref{tab:comparison_variables} compares the variables involved in each formulation and highlights the structural differences between the two formulations (nomenclature in Type is taken from \cite{ham2020transfer}). Both models share the same number of transfer, operation, and sequencing variables. However, they differ in how transportation and machine assignment decisions are represented. In the arc-based formulation, transfer decisions are decomposed into arc selection variables $x_{oa}$ and leg-level transbot assignment variables $x_{olv}$. As a result, the number of transbot-assignment variables scales with $\sum_{o \in \mathcal{O}} \sum_{l \in \mathcal{L}_o} |\mathcal{V}_{z_l}|$, which explicitly depends on the number of legs induced by each operation. In contrast, the operation-embedded formulation eliminates explicit arc-selection variables and defines transbot-assignment variables directly at the leg level ($x_{lv}$), leading to a flatter structure with cardinality $\sum_{l \in \mathcal{L}} |\mathcal{V}_{z_l}|$. Similarly, machine-assignment decisions are embedded within arc-level variables $y_a$, whose number equals the feasible arcs per operation. Overall, the arc-based formulation provides a more explicit representation of routing choices, while the operation-embedded formulation integrates routing and processing decisions more tightly. This structural distinction affects model size and propagation behavior.

\noindent \textbf{Asymptotic Size Comparison.} Let $|\mathcal{O}|$ denote the number of operations, $|\mathcal{M}|$ the number of machines, $|\mathcal{V}|$ the number of transbots, and $|\mathcal{Z}|$ be the number of zones. Assume each zone contains on average $|\mathcal{V}|/|\mathcal{Z}|$ transbots. In the arc-based formulation, the number of transbot-assignment variables is
\[
\sum_{o \in \mathcal{O}} \sum_{l \in \mathcal{L}_o} |\mathcal{V}_{z_l}|.
\]
Since each operation induces at most two legs, i.e., $|\mathcal{L}_o| \le 2$, the total number of such variables is bounded by
\[
O\!\left(|\mathcal{O}| \cdot \frac{|\mathcal{V}|}{|\mathcal{Z}|}\right),
\]
under uniform zone distribution. Additionally, arc-selection variables scale as $\sum_{o \in \mathcal{O}} |\mathcal{A}_o|$, which in the worst case is $O(|\mathcal{O}| \cdot |\mathcal{M}|^2)$, though typically much smaller due to routing and zoning restrictions. In the operation-embedded formulation, transbot-assignment variables scale as
\[
\sum_{l \in \mathcal{L}} |\mathcal{V}_{z_l}|,
\]
which is also bounded by 
\[
O\!\left(|\mathcal{O}| \cdot \frac{|\mathcal{V}|}{|\mathcal{Z}|}\right),
\]
since $|\mathcal{L}| \le 2|\mathcal{O}|$. However, machine-assignment variables $y_a$ grow with $\sum_{o \in \mathcal{O}} |\mathcal{A}_o|$, embedding routing and processing decisions into a single layer. Therefore, both formulations scale linearly in the number of operations with respect to transportation variables, while machine-assignment variables may scale quadratically in machine flexibility through $|\mathcal{A}_o|$. The main structural difference lies in the explicit presence of arc-selection variables in the arc-based formulation. This distinction affects the memory footprint and the propagation structure rather than the asymptotic order. The two formulations produce the same makespan, as shown below.

\begin{prop}\label{equivalence}
The arc-based and operation-embedded formulations produce the same optimal makespan.  
\end{prop}
\begin{proof}
Both formulations minimize the same objective function
\[
\min \max_{o \in \mathcal{O}} \mathtt{endOf}(y_o),
\]
where $y_o$ denotes the execution interval of operation $o$. 
It therefore suffices to prove that the two formulations induce identical feasible regions.

\medskip
\noindent
\textit{Step 1: Processing equivalence.}
In both formulations, each operation $o \in \mathcal{O}$ selects exactly one feasible machine from $\mathcal{M}_o$. The execution interval $y_o$ is assigned to the selected machine and has a duration equal to the corresponding processing time $p_{om}$. Although machine assignment is encoded differently (via $y_m$ in the arc-based model and via $y_a$ in the operation-embedded model), the induced execution interval $y_o$ has identical machine assignment and duration in both formulations. Thus, the processing component of any feasible schedule is identical.

\medskip
\noindent
\textit{Step 2: Transfer structure equivalence.}
For each operation, transportation between consecutive machines is determined by the selected pickup and drop-off machines. In the arc-based formulation, a transfer is represented by a selected arc $a$ that decomposes into an ordered set of legs $\mathcal{L}_a$, with $|\mathcal{L}_a| \in \{0,1,2\}$ depending on zoning. In the special case where consecutive operations are processed on the same machine, no transfer leg is generated in either formulation. In the operation-embedded formulation, the same leg structure is represented directly, without an explicit arc-selection layer.

For any pair of selected machines $(m,m')$ that is feasible under the routing and zoning rules, the induced leg set and associated travel times are identical in both formulations, since both rely on the same underlying travel-time matrix and zoning structure.

\medskip
\noindent
\textit{Step 3: Resource and temporal consistency.}
Both formulations enforce:
(i) precedence between transfer and processing intervals,
(ii) sequencing constraints on machines,
(iii) sequencing constraints on transbots,
and (iv) ordered execution of legs for inter-zone transfers.
These constraints impose identical temporal and resource-feasibility conditions on the set of intervals.

\medskip
\noindent
\textit{Step 4: Bidirectional mapping of feasible solutions.}
Let $\mathcal{F}^{\text{Arc}}$ and $\mathcal{F}^{\text{Emb}}$ denote the feasible schedule sets of the arc-based and operation-embedded formulations, respectively.

\smallskip
\noindent
($\subseteq$) Consider any feasible schedule in $\mathcal{F}^{\text{Arc}}$. Removing the auxiliary arc-selection variables while retaining the selected machines, leg assignments, start times, and durations yields a schedule that satisfies all constraints of the operation-embedded formulation. Hence, $\mathcal{F}^{\text{Arc}} \subseteq \mathcal{F}^{\text{Emb}}$.

\smallskip
\noindent
($\supseteq$) Conversely, consider any feasible schedule in $\mathcal{F}^{\text{Emb}}$. The selected machine pair for each operation uniquely determines a corresponding arc consistent with zoning and routing logic. Introducing the associated arc-selection variables produces a schedule that satisfies all arc-based constraints without altering any timing decisions. Hence, $\mathcal{F}^{\text{Emb}} \subseteq \mathcal{F}^{\text{Arc}}$.

\medskip

Since $\mathcal{F}^{\text{Arc}} = \mathcal{F}^{\text{Emb}}$, the two formulations define identical feasible regions and therefore share the same optimal objective value.
\end{proof}

Although the two formulations are theoretically equivalent with respect to feasible regions and optimal objective values, their internal search trees may differ substantially due to differences in propagation strength and domain filtering. This distinction is central to computational performance and motivates the experimental comparison in Section \ref{section:7}.

\subsection{Solution Methodology}\label{section:methodology}

Each arc represents a material transfer from a pickup machine $m$ to a drop-off machine $m'$. Recall that, to incorporate zoning, arcs are decomposed into one or two legs depending on the zones of their endpoints:

\allowdisplaybreaks
\begin{itemize}
    \item \textbf{Different zones ($z_m \neq z_{m'}$):} the arc is split into two legs $l_1, l_2 \in \mathcal{L}_a$ such that:
    \[
    \begin{aligned}
        l_1 &: m_{l_1} = m, \quad m^{'}_{l_1} = H, \quad z_{v_{l_1}} = z_m,\\
        l_2 &: m_{l_2} = H, \quad m^{'}_{l_2} = m', \quad z_{v_{l_2}} = z_{m'},
    \end{aligned}
    \]
    where $v_{l_k}$ denotes the transbot assigned to leg $l_k$.
    
    \item \textbf{Same zone ($z_m = z_{m'}$):} the arc consists of a single leg $l \in \mathcal{L}_a$:
    \[
         m_{l} = m, \quad m^{'}_{l} = m', \quad z_{v_l} = z_m.
    \]
\end{itemize}

From a modeling perspective, this decomposition ensures that each transbot is assigned only to legs within its accessible zone and that material transfers between zones occur via a designated handoff point, thereby preventing routing conflicts and ensuring exclusive leg occupancy. From a practical perspective, a machine executes a single operation at a time, implying no two operations share the same arc at a given timestamp. 

The movement of each transbot must be coordinated to avoid conflicts over shared legs of the transportation network. To achieve this, the paper adopts a book-and-release strategy within the CP models. When a transbot is scheduled to traverse a leg $l \in \mathcal{L}_a, \ a \in \mathcal{A}$, the corresponding interval variable books exclusive use of that leg for the duration of the traversal. This prevents other transbots from accessing the same leg during the booked interval.  Once the traversal is completed, the leg is released and becomes available to other transbots. The strategy is naturally enforced using the \texttt{noOverlap} Constraints \eqref{c8 - noOverlap(vehicles)} and \eqref{m2c10 - noOverlap(vehicles)} on leg resources, ensuring collision-free routing without explicitly modeling continuous transbot movement.

Since each identifier $i_{l} \in \mathcal{I}, \ l \in\mathcal{L}_{a}, \ a \in \mathcal{A}$ maps to the paired pickup and drop-off machines for a leg (including the handoff point), each matrix $T$ entry represents the travel time from the drop-off machine associated with the row index to the pickup machine associated with the column index. This ensures that the time difference between consecutive transfer intervals reflects the actual travel time required for the transbot to move from the drop-off of one leg to the pickup of the next, rather than the transfer interval duration itself. Storing indices enables the model to retrieve the appropriate transfer time for each consecutive movement. This modeling approach provides a discrete-time abstraction of continuous vehicle movement while preserving routing feasibility and travel-time accuracy.

\section{Computational Results}\label{section:7}

This section presents the experimental design, computational performance, and sensitivity analysis. The section concludes with a discussion of the managerial insights derived from the results, with an emphasis on the practical benefits.

\subsection{Experimental design}

The CP formulations are implemented using IBM CP Optimizer. All computations are performed on a 40-core machine with 384 GB of RAM, running Oracle Linux Server 7.7 on a 3.20 GHz Intel® Core™ i7-8700 processor. For evaluation, real-time runtime is reported to assess computational performance across all test instances. A maximum runtime of 600 seconds was applied to all computational runs. This section describes the test instances and acceleration strategies. 

\noindent \textbf{Test Instances.} The approaches in this paper are evaluated on two instance classes: small (S) and medium (M). The small-scale instances are taken from \cite{deroussi2010simultaneous}, which extend the classical \cite{bilge1995time} benchmark by adding machines to suit settings with transbots. For medium-scale experiments, the well-known \emph{HUdata} benchmark \citep{hurink1994tabu} is used. This benchmark includes four families of flexible job shop instances (\emph{sdata}, \emph{edata}, \emph{rdata}, \emph{vdata}) characterized by varying levels of routing flexibility. \textit{sdata} models cases with no flexibility, representing the classical JSP case, \textit{edata} models cases with partial flexibility where few operations can be assigned to multiple machines, \textit{rdata} models those where multiple operations can be assigned to multiple machines, and \textit{vdata} reflects full flexibility with $|\mathcal{M}|/2$ available machines on average per operation. Transfer times are generated from \textit{Layout 1} of \cite{ham2020transfer}, where times are uniformly distributed from [20,40]. The considered benchmark instances are adapted to incorporate zoning and to mimic the industrial case study that inspired this research. The instance generation is highlighted in Appendix \ref{B}.

\noindent \textbf{Acceleration Strategies.} As a first acceleration strategy, the problem structure is exploited by initially solving the classical FJSP. This relaxation ignores zoning, routing, and transfer operations, and therefore provides a computationally inexpensive surrogate of the FJSPT-H. The objective value constitutes a valid lower bound for the FJSPT-H, since introducing transbot movements can only increase the makespan. 

Beyond warm-starting, further computational gains can be obtained by carefully tuning CP solver parameters. In particular, the number of workers controls parallel exploration of the search tree, allowing multiple threads to evaluate independent branches simultaneously. Theoretically, increasing the number of workers can accelerate the discovery of high-quality solutions by covering a larger portion of the solution space in parallel. Still, it also introduces coordination and memory overhead. Empirically, for larger or highly constrained instances of FJSPT-H, increasing the number of workers reduces solving time and improves convergence to optimal or near-optimal solutions. Adjusting this parameter balances extraction time and engine computation, thereby improving solver efficiency without modifying the model itself.

\subsection{Computational Performance}

To measure the performance of our model, the following methods are compared:

\begin{itemize}
    \item \textbf{CP2-H}: The FJSPT model from \cite{ham2020transfer}.
    \item \textbf{CP2-H$^{+}$}: The FJSPT-H without zoning, without handoff points, and with homogeneous transbots. CP2-H$^{+}$ is feature-equivalent to CP2-H but adopts an arc-based formulation.
    \item \textbf{CP2-H$^{++}$}: The FJSPT-H without zoning, without handoff points, and with homogeneous transbots. CP2-H$^{++}$ is feature-equivalent to CP2-H, but it adopts the operation-embedded formulation.
    \item \textbf{CP3}: The arc-based formulation for FJSPT-H proposed in this paper.
    \item \textbf{CP4}: The operation-embedded formulation for FJSPT-H proposed in this paper.
\end{itemize}

For each formulation, the resulting makespan (CMAX) and the execution time (Time) in seconds are reported. Table \ref{tab:Results on Small Instances} compares the performance of CP2-H, CP2-H$^{+}$, CP2-H$^{++}$, CP3, and CP4 across the ten small FJSPT instances under the configuration: two zones, two transbots. 

The results on the small instances show a clear and consistent advantage for the models built on our new formulations. CP2-H is significantly slower, with an average runtime of 28.15s, more than an order of magnitude higher than that of the strengthened variants. In contrast, CP2-H$^{+}$ and CP2-H$^{++}$ solve the same feature-equivalent problem in only 1.74s and 1.57s on average, respectively. The improvement is particularly evident in instances such as FJSP7, where the runtime drops from 251.20s for CP2-H to only a few seconds for the alternative formulations. This demonstrates that tighter representations, achieved through additional global constraints and a more structured modeling of transfers, substantially improve propagation and search efficiency, independently of zoning or handoff features.

In terms of solution quality, the models yield comparable makespans. When the full FJSPT-H features are activated, CP3 and CP4 exhibit very similar computational behavior, with average runtimes of 1.41s and 1.28s, respectively. The comparable performance of these models indicates that incorporating zoning, handoff points, and heterogeneous transbots does not significantly increase the computational complexity at this scale. Overall, these results confirm that the formulation itself is the main driver of the observed computational gains.

\begin{table}[t!]
  \centering
  \caption{Results on Small Instances}
    \begin{tabular}{l|rr|rr|rr|rr|rr}
    \toprule
    \multirow{2}[3]{*}{Instance} & \multicolumn{2}{c|}{CP2-H} & \multicolumn{2}{c|}{CP2-H$^+$} & \multicolumn{2}{c|}{CP2-H$^{++}$} & \multicolumn{2}{c|}{CP3} & \multicolumn{2}{c}{CP4} \\
\cmidrule{2-11}          & CMAX  & Time  & CMAX  & Time  & CMAX  & Time  & CMAX  & Time  & CMAX  & Time \\
    \midrule
    FJSP1 & 134   & 5.30  & 134   & 1.24  & 134   & 1.11  & 144   & 1.79  & 144   & 1.27 \\
    FJSP2 & 114   & 5.00  & 116   & 1.22  & 116   & 1.24  & 120   & 0.60  & 120   & 0.72 \\
    FJSP3 & 120   & 1.10  & 120   & 0.18  & 120   & 0.06  & 132   & 0.46  & 132   & 0.53 \\
    FJSP4 & 114   & 5.40  & 114   & 2.46  & 114   & 1.90  & 130   & 1.31  & 130   & 0.81 \\
    FJSP5 & 94    & 0.20  & 90    & 0.15  & 90    & 0.05  & 98    & 0.52  & 98    & 0.34 \\
    FJSP6 & 138   & 1.40  & 136   & 1.25  & 136   & 1.19  & 148   & 1.96  & 148   & 1.14 \\
    FJSP7 & 108   & 251.20 & 108   & 3.96  & 108   & 2.00  & 120   & 2.75  & 120   & 3.60 \\
    FJSP8 & 178   & 1.00  & 178   & 1.28  & 178   & 1.45  & 180   & 0.87  & 180   & 1.19 \\
    FJSP9 & 144   & 4.10  & 142   & 1.02  & 142   & 0.67  & 152   & 1.07  & 152   & 0.53 \\
    FJSP10 & 174   & 6.80  & 174   & 4.67  & 174   & 6.05  & 178   & 2.80  & 178   & 2.67 \\
    \midrule
    \textbf{Avg} & \textbf{131} & \textbf{28.15} & \textbf{131} & \textbf{1.74} & \textbf{131} & \textbf{1.57} & \textbf{140} & \textbf{1.41} & \textbf{140} & \textbf{1.28} \\
    \bottomrule
    \end{tabular}%
  \label{tab:Results on Small Instances}%
\end{table}%

Table \ref{tab:Results on Medium Instances} reports the average makespans obtained for the medium-scale instances in each HUdata class. Across these instances, the feature-equivalent models CP2-H, CP2-H$^{+}$, and CP2-H$^{++}$ remain closely aligned in performance. CP2-H achieves an average makespan of 2,413.50, while CP2-H$^{+}$ and CP2-H$^{++}$ obtain averages of 2,396.68 and 2,407.52, respectively. This confirms that the alternative formulations preserve solution quality in the absence of zoning and handoff constraints and benefit from the structural advantages observed in the small-instance experiments.

In contrast, CP3 and CP4 yield noticeably larger makespans across all classes, with average values of 3,127.39 and 3,345.80, respectively. This increase reflects the additional operational constraints introduced in the full FJSPT-H setting, including zoning restrictions, handoff points, and heterogeneous transbots, which reduce routing flexibility and make the coordination of machine–robot interactions more challenging. The performance gap is particularly pronounced for the edata and rdata sets, where transfer-time heterogeneity and routing restrictions amplify the complexity of the transportation decisions. Overall, these results indicate that although the proposed formulations remain robust, the additional realism introduced in CP3 and CP4 naturally yields larger optimal makespans.

\begin{table}[t!]
  \centering
  \caption{Results on Medium Instances}
    \begin{tabular}{l|c|c|c|c|c}
    \toprule
    Instance & CP2-H & CP2-H$^{+}$ & CP2-H$^{++}$ & CP3   & CP4 \\
    \midrule
    sdata & 2{,}795.10 & 2{,}882.33 & 2{,}844.98 & 4{,}131.95 & 4{,}072.15 \\
    \midrule
    edata & 2{,}920.50 & 2{,}780.25 & 2{,}779.75 & 3{,}934.08 & 3{,}926.95 \\
    \midrule
    rdata & 2{,}398.90 & 2{,}431.78 & 2{,}383.30 & 3{,}179.93 & 3{,}194.95 \\
    \midrule
    vdata & 1{,}921.10 & 1{,}978.03 & 2{,}059.50 & 2{,}268.18 & 2{,}915.50 \\
    \midrule
    \textbf{Avg} & \textbf{2{,}508.90} & \textbf{2{,}518.09} & \textbf{2{,}516.88} & \textbf{3{,}378.53} & \textbf{3{,}527.38} \\
    \bottomrule
    \end{tabular}%
  \label{tab:Results on Medium Instances}%
\end{table}%

Figure \ref{fig:L01_evolution} reports the evolution of CMAX over time for instance la01 under CP2-H$^{+}$, CP2-H$^{++}$, CP3, and CP4 for the four variants sdata, edata, rdata, and vdata. In both models, the curves exhibit the typical rapid-improvement phase during the earliest iterations (0–100s), followed by a plateau as the search converges. For CP2-H$^{+}$ and CP2-H$^{++}$, all four instance classes stabilize within roughly 200s, with vdata converging fastest and edata maintaining the highest makespan values. Under CP3 and CP4, a similar pattern appears, but final CMAX values are systematically higher, and the convergence is slightly slower, particularly for edata and sdata, which remain well above the other variants throughout the entire horizon.

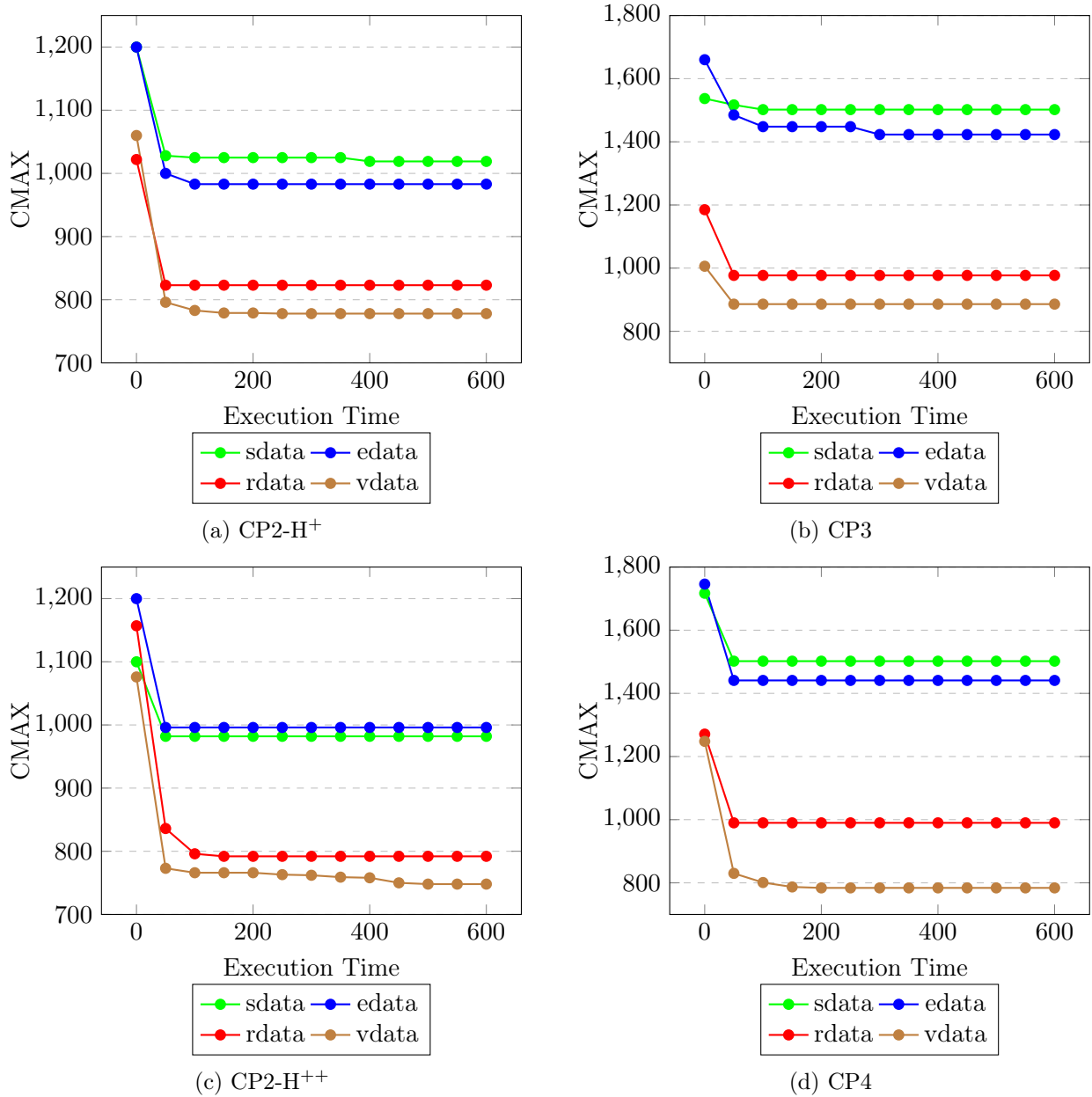
\begin{figure}[t!]
    \centering
    \begin{subfigure}[b]{0.48\textwidth}
        \centering
        \begin{tikzpicture}
        \begin{axis}[
            width=\textwidth,
            xlabel={Execution Time},
            ylabel={CMAX},
            ymin=700, ymax=1250,
            yticklabel style={/pgf/number format/fixed},
            xtick={0,200,400,600},
            legend style={at={(0.5,-0.2)}, anchor=north, legend columns=2},
            ymajorgrids=true,
            grid style=dashed,
        ]
        \addplot[
            color=green,
            mark=*,
            thick
        ] coordinates {
            (0,1200)
            (50,1028)
            (100,1025)
            (150,1025)
            (200,1025)
            (250,1025)
            (300,1025)
            (350,1025)
            (400,1019)
            (450,1019)
            (500,1019)
            (550,1019)
            (600,1019)
        };
        \addlegendentry{sdata}
        \addplot[
            color=blue,
            mark=*,
            thick
        ] coordinates {
            (0,1200)
            (50,1000)
            (100,983)
            (150,983)
            (200,983)
            (250,983)
            (300,983)
            (350,983)
            (400,983)
            (450,983)
            (500,983)
            (550,983)
            (600,983)
        };
        \addlegendentry{edata}
        \addplot[
            color=red,
            mark=*,
            thick
        ] coordinates {
            (0,1022)
            (50,823)
            (100,823)
            (150,823)
            (200,823)
            (250,823)
            (300,823)
            (350,823)
            (400,823)
            (450,823)
            (500,823)
            (550,823)
            (600,823)
        };
        \addlegendentry{rdata}
        \addplot[
            color=brown,
            mark=*,
            thick
        ] coordinates {
            (0,1060)
            (50,796)
            (100,783)
            (150,779)
            (200,779)
            (250,778)
            (300,778)
            (350,778)
            (400,778)
            (450,778)
            (500,778)
            (550,778)
            (600,778)
        };
        \addlegendentry{vdata}
        \end{axis}
        \end{tikzpicture}
        \caption{CP2-H$^+$}
        \label{fig:L01_evolution_CP2-H}
    \end{subfigure}
    \hfill
    \begin{subfigure}[b]{0.48\textwidth}
        \centering
        \begin{tikzpicture}
        \begin{axis}[
            width=\textwidth,
            xlabel={Execution Time},
            ylabel={CMAX},
            xtick={0,200,400,600},
            ymin=700, ymax=1800,
            legend style={at={(0.5,-0.2)}, anchor=north, legend columns=2},
            ymajorgrids=true,
            grid style=dashed,
        ]
        \addplot[
            color=green,
            mark=*,
            thick
        ] coordinates {
            (0,1537)
            (50,1517)
            (100,1502)
            (150,1502)
            (200,1502)
            (250,1502)
            (300,1502)
            (350,1502)
            (400,1502)
            (450,1502)
            (500,1502)
            (550,1502)
            (600,1502)
        };
        \addlegendentry{sdata}
        \addplot[
            color=blue,
            mark=*,
            thick
        ] coordinates {
            (0,1660)
            (50,1485)
            (100,1448)
            (150,1448)
            (200,1448)
            (250,1448)
            (300,1423)
            (350,1423)
            (400,1423)
            (450,1423)
            (500,1423)
            (550,1423)
            (600,1423)
        };
        \addlegendentry{edata}
        \addplot[
            color=red,
            mark=*,
            thick
        ] coordinates {
            (0,1185)
            (50,977)
            (100,977)
            (150,977)
            (200,977)
            (250,977)
            (300,977)
            (350,977)
            (400,977)
            (450,977)
            (500,977)
            (550,977)
            (600,977)
        };
        \addlegendentry{rdata}
        \addplot[
            color=brown,
            mark=*,
            thick
        ] coordinates {
            (0,1006)
            (50,886)
            (100,886)
            (150,886)
            (200,886)
            (250,886)
            (300,886)
            (350,886)
            (400,886)
            (450,886)
            (500,886)
            (550,886)
            (600,886)
        };
        \addlegendentry{vdata}
        \end{axis}
        \end{tikzpicture}
        \caption{CP3}
        \label{fig:L01_evolution_CP3}
    \end{subfigure}
    \begin{subfigure}[b]{0.48\textwidth}
        \centering
        \begin{tikzpicture}
        \begin{axis}[
            width=\textwidth,
            xlabel={Execution Time},
            ylabel={CMAX},
            ymin=700, ymax=1250,
            yticklabel style={/pgf/number format/fixed},
            xtick={0,200,400,600},
            legend style={at={(0.5,-0.2)}, anchor=north, legend columns=2},
            ymajorgrids=true,
            grid style=dashed,
        ]
        \addplot[
            color=green,
            mark=*,
            thick
        ] coordinates {
            (0,1100)
            (50,982)
            (100,982)
            (150,982)
            (200,982)
            (250,982)
            (300,982)
            (350,982)
            (400,982)
            (450,982)
            (500,982)
            (550,982)
            (600,982)
        };
        \addlegendentry{sdata}
        \addplot[
            color=blue,
            mark=*,
            thick
        ] coordinates {
            (0,1200)
            (50,996)
            (100,996)
            (150,996)
            (200,996)
            (250,996)
            (300,996)
            (350,996)
            (400,996)
            (450,996)
            (500,996)
            (550,996)
            (600,996)
        };
        \addlegendentry{edata}
        \addplot[
            color=red,
            mark=*,
            thick
        ] coordinates {
            (0,1157)
            (50,836)
            (100,796)
            (150,792)
            (200,792)
            (250,792)
            (300,792)
            (350,792)
            (400,792)
            (450,792)
            (500,792)
            (550,792)
            (600,792)
        };
        \addlegendentry{rdata}
        \addplot[
            color=brown,
            mark=*,
            thick
        ] coordinates {
            (0,1076)
            (50,773)
            (100,766)
            (150,766)
            (200,766)
            (250,763)
            (300,762)
            (350,759)
            (400,758)
            (450,750)
            (500,748)
            (550,748)
            (600,748)
        };
        \addlegendentry{vdata}
        \end{axis}
        \end{tikzpicture}
        \caption{CP2-H$^{++}$}
        \label{fig:L01_evolution_CP2-HH}
    \end{subfigure}
    \hfill
    \begin{subfigure}[b]{0.48\textwidth}
        \centering
        \begin{tikzpicture}
        \begin{axis}[
            width=\textwidth,
            xlabel={Execution Time},
            ylabel={CMAX},
            xtick={0,200,400,600},
            ymin=700, ymax=1800,
            legend style={at={(0.5,-0.2)}, anchor=north, legend columns=2},
            ymajorgrids=true,
            grid style=dashed,
        ]
        \addplot[
            color=green,
            mark=*,
            thick
        ] coordinates {
            (0,1717)
            (50,1502)
            (100,1502)
            (150,1502)
            (200,1502)
            (250,1502)
            (300,1502)
            (350,1502)
            (400,1502)
            (450,1502)
            (500,1502)
            (550,1502)
            (600,1502)
        };
        \addlegendentry{sdata}
        \addplot[
            color=blue,
            mark=*,
            thick
        ] coordinates {
            (0,1746)
            (50,1441)
            (100,1441)
            (150,1441)
            (200,1441)
            (250,1441)
            (300,1441)
            (350,1441)
            (400,1441)
            (450,1441)
            (500,1441)
            (550,1441)
            (600,1441)
        };
        \addlegendentry{edata}
        \addplot[
            color=red,
            mark=*,
            thick
        ] coordinates {
            (0,1271)
            (50,990)
            (100,990)
            (150,990)
            (200,990)
            (250,990)
            (300,990)
            (350,990)
            (400,990)
            (450,990)
            (500,990)
            (550,990)
            (600,990)
        };
        \addlegendentry{rdata}
        \addplot[
            color=brown,
            mark=*,
            thick
        ] coordinates {
            (0,1248)
            (50,830)
            (100,801)
            (150,787)
            (200,784)
            (250,784)
            (300,784)
            (350,784)
            (400,784)
            (450,784)
            (500,784)
            (550,784)
            (600,784)
        };
        \addlegendentry{vdata}
        \end{axis}
        \end{tikzpicture}
        \caption{CP4}
        \label{fig:L01_evolution_CP4}
    \end{subfigure}
    \caption{Evolution of CMAX with Execution Time for Instances la01}
    \label{fig:L01_evolution}
\end{figure}

Two insights emerge from these trends. First, CP2-H$^{+}$ and CP2-H$^{++}$ consistently achieve lower makespan values more quickly than CP3 and CP4, highlighting the added complexity when considering zoning and handoff. The strong early reductions in CMAX across all variants suggest that both models rapidly exploit simple improvement opportunities, but CP3 and CP4’s additional structural constraints restrict feasible moves and slow convergence. Second, the relative ordering of sdata, edata, rdata, and vdata persists across formulations, indicating that instance flexibility is a dominant driver of difficulty. The steeper and longer convergence tails of CP3 and CP4 reflect the increased combinatorial complexity introduced by heterogeneity and zoning, which leads to higher final makespans within the fixed time limit.

\subsection{Sensitivity Analysis}

In this section, the impact of the number of transbots and the number of zones on both small and medium-scale instances is evaluated. Then, the impact of the layout on medium-scale instances is assessed.

Table~\ref{tab:tranbot_impact} reports the results of a sensitivity analysis where the number of zones is kept constant at two while the number of transbots is progressively increased from two to ten. For each configuration, the table reports the resulting makespan and execution time for both proposed formulations, CP3 and CP4. A first clear observation is that increasing the number of transbots significantly reduces the makespan when transportation resources are initially scarce. Moving from two to four transbots reduces the average makespan from 140 to 137, while increasing the fleet to six transbots further decreases the average makespan to 129. This improvement reflects the reduction of transportation waiting times between consecutive operations. With a limited number of transbots, jobs frequently wait for a transbot to become available, delaying the start of subsequent machine operations and increasing the overall makespan.

However, beyond six transbots, additional transbots no longer provide further improvements in solution quality. The average makespan stabilizes at approximately 129 for configurations with six, eight, and ten transbots. This plateau indicates that transportation ceases to be the primary bottleneck once sufficient transbot capacity is available. At this stage, the schedule's critical path is primarily determined by machine-processing constraints rather than by material-handling delays, so additional transportation capacity yields diminishing returns.

The results also reveal that both formulations have identical makespan values across all configurations, confirming their equivalence in solution quality. In terms of computational performance, however, CP4 consistently requires less execution time than CP3. This difference becomes more pronounced as the number of transbots increases, suggesting that the operation-embedded formulation scales more efficiently with larger transportation fleets. Overall, these results highlight the importance of balancing transportation capacity with machine resources, and they suggest that a moderate number of transbots is sufficient to achieve near-optimal system performance in zoned robotic flexible job shops.

\begin{table}[t!]
  \centering
  \scalebox{0.55}{
  \caption{Transbot Impact Small Instances} 
    \begin{tabular}{l|rr|rr|rr|rr|rr|rr|rr|rr|rr|rr}
    \toprule
    Zones & \multicolumn{4}{c|}{2}        & \multicolumn{4}{c|}{2}        & \multicolumn{4}{c|}{2}        & \multicolumn{4}{c|}{2}        & \multicolumn{4}{c}{2} \\
    \midrule
    Transbots & \multicolumn{4}{c|}{2}        & \multicolumn{4}{c|}{4}        & \multicolumn{4}{c|}{6}        & \multicolumn{4}{c|}{8}        & \multicolumn{4}{c}{10} \\
    \midrule
    \multirow{2}[4]{*}{Instance} & \multicolumn{2}{c|}{CP3} & \multicolumn{2}{c|}{CP4} & \multicolumn{2}{c|}{CP3} & \multicolumn{2}{c|}{CP4} & \multicolumn{2}{c|}{CP3} & \multicolumn{2}{c|}{CP4} & \multicolumn{2}{c|}{CP3} & \multicolumn{2}{c|}{CP4} & \multicolumn{2}{c|}{CP3} & \multicolumn{2}{c}{CP4} \\
\cmidrule{2-21}          & CMAX  & Time  & CMAX  & Time  & CMAX  & Time  & CMAX  & Time  & CMAX  & Time  & CMAX  & Time  & CMAX  & Time  & CMAX  & Time  & CMAX  & Time  & CMAX  & Time \\
    \midrule
    FJSP1 & 144   & 1.79  & 144   & 1.27  & 138   & 0.59  & 138   & 0.13  & 132   & 0.87  & 132   & 0.18  & 132   & 1.02  & 132   & 0.20  & 132   & 1.41  & 132   & 0.20 \\
    \midrule
    FJSP2 & 120   & 0.60  & 120   & 0.72  & 112   & 0.41  & 112   & 0.07  & 110   & 0.63  & 110   & 0.14  & 110   & 0.69  & 110   & 0.13  & 110   & 0.88  & 110   & 0.14 \\
    \midrule
    FJSP3 & 132   & 0.46  & 132   & 0.53  & 136   & 0.46  & 136   & 0.07  & 126   & 1.03  & 126   & 0.55  & 126   & 1.26  & 126   & 0.71  & 126   & 2.11  & 126   & 0.56 \\
    \midrule
    FJSP4 & 130   & 1.31  & 130   & 0.81  & 126   & 0.89  & 126   & 0.20  & 118   & 1.49  & 118   & 0.50  & 118   & 1.64  & 118   & 0.58  & 118   & 2.40  & 118   & 0.90 \\
    \midrule
    FJSP5 & 98    & 0.52  & 98    & 0.34  & 100   & 0.32  & 100   & 0.06  & 90    & 0.44  & 90    & 0.08  & 90    & 0.56  & 90    & 0.11  & 90    & 0.65  & 90    & 0.09 \\
    \midrule
    FJSP6 & 148   & 1.96  & 148   & 1.14  & 142   & 0.75  & 142   & 0.30  & 138   & 4.42  & 138   & 1.30  & 138   & 6.70  & 138   & 2.22  & 138   & 8.46  & 138   & 2.78 \\
    \midrule
    FJSP7 & 120   & 2.75  & 120   & 3.60  & 108   & 0.69  & 108   & 0.14  & 98    & 2.00  & 98    & 0.69  & 98    & 2.12  & 98    & 0.88  & 98    & 2.76  & 98    & 1.17 \\
    \midrule
    FJSP8 & 180   & 0.87  & 180   & 1.19  & 188   & 2.50  & 188   & 1.91  & 180   & 3.01  & 180   & 2.23  & 180   & 5.04  & 180   & 2.00  & 180   & 7.19  & 180   & 4.18 \\
    \midrule
    FJSP9 & 152   & 1.07  & 152   & 0.53  & 148   & 0.51  & 148   & 0.22  & 138   & 1.95  & 138   & 1.04  & 138   & 3.20  & 138   & 0.95  & 138   & 4.79  & 138   & 1.33 \\
    \midrule
    FJSP10 & 178   & 2.80  & 178   & 2.67  & 176   & 1.42  & 176   & 2.17  & 164   & 5.60  & 164   & 1.01  & 164   & 6.80  & 164   & 1.77  & 164   & 4.86  & 164   & 2.20 \\
    \midrule
    \textbf{Avg} & \textbf{140} & \textbf{1.41} & \textbf{140} & \textbf{1.28} & \textbf{137} & \textbf{0.85} & \textbf{137} & \textbf{0.53} & \textbf{129} & \textbf{2.14} & \textbf{129} & \textbf{0.77} & \textbf{129} & \textbf{2.90} & \textbf{129} & \textbf{0.96} & \textbf{129} & \textbf{3.55} & \textbf{129} & \textbf{1.36} \\
    \bottomrule
    \end{tabular}%
  \label{tab:tranbot_impact}%
}
\end{table}%

Table~\ref{tab:zone_impact} analyzes the impact of zoning on system performance while keeping the number of transbots fixed at four. The number of zones is progressively varied from 1 to 4 to evaluate how spatial partitioning affects coordination between machines and transportation resources. For each configuration, the resulting makespan and execution time are reported for both formulations.

The results show that zoning has a significant impact on the overall system performance. When all machines belong to a single zone, the average makespan is 125, which corresponds to the most flexible transportation configuration. In this setting, all transbots can serve any machine, and material transfers can be executed directly without intermediate handoffs. As a result, transportation delays are minimized, and jobs can be transferred more efficiently between consecutive operations.

Introducing zoning constraints increases the makespan, as observed when moving from one zone to two zones, where the average makespan rises from 125 to 137. This degradation is primarily due to the reduced routing flexibility induced by zoning. When machines are partitioned into separate zones, inter-zone transfers must be coordinated through handoff points, which introduces additional synchronization requirements and limits the ability of transbots to respond dynamically to transportation requests.

Interestingly, increasing the number of zones from two to four partially mitigates this effect, reducing the average makespan to 130. This behavior can be explained by the smaller spatial scope of each zone, which shortens intra-zone travel distances and reduces local transportation congestion. Overall, these results highlight the trade-off introduced by zoning policies: while zoning can simplify traffic management and reduce local travel distances, it may also restrict routing flexibility and increase coordination overhead. Consequently, an appropriate balance between zoning structure and transportation flexibility is critical for achieving efficient scheduling in robotic flexible job shops.

\begin{table}[t!]
  \centering
  \scalebox{0.8}{
  \caption{Zoning Impact Small Instances}
    \begin{tabular}{l|rr|rr|rr|rr|rr|rr}
    \toprule
    Zones & \multicolumn{4}{c|}{1}        & \multicolumn{4}{c|}{2}        & \multicolumn{4}{c}{4} \\
    \midrule
    Transbots & \multicolumn{4}{c|}{4}        & \multicolumn{4}{c|}{4}        & \multicolumn{4}{c}{4} \\
    \midrule
    \multirow{2}[4]{*}{Instance} & \multicolumn{2}{c|}{CP3} & \multicolumn{2}{c|}{CP4} & \multicolumn{2}{c|}{CP3} & \multicolumn{2}{c|}{CP4} & \multicolumn{2}{c|}{CP3} & \multicolumn{2}{c}{CP4} \\
\cmidrule{2-13}          & CMAX  & Time  & CMAX  & Time  & CMAX  & Time  & CMAX  & Time  & CMAX  & Time  & CMAX  & Time \\
    \midrule
    FJSP1 & 132   & 0.45  & 132   & 0.17  & 138   & 0.59  & 138   & 0.13  & 132   & 0.6   & 132   & 0.14 \\
    \midrule
    FJSP2 & 104   & 0.30   & 104   & 0.06  & 112   & 0.41  & 112   & 0.07  & 110   & 0.41  & 110   & 0.11 \\
    \midrule
    FJSP3 & 120   & 0.35  & 120   & 0.05  & 136   & 0.46  & 136   & 0.07  & 126   & 0.89  & 126   & 0.32 \\
    \midrule
    FJSP4 & 102   & 0.46  & 102   & 0.06  & 126   & 0.89  & 126   & 0.20   & 118   & 1.06  & 118   & 0.82 \\
    \midrule
    FJSP5 & 90    & 0.23  & 90    & 0.02  & 100   & 0.32  & 100   & 0.06  & 90    & 0.29  & 90    & 0.08 \\
    \midrule
    FJSP6 & 132   & 0.55  & 132   & 0.15  & 142   & 0.75  & 142   & 0.30   & 138   & 1.83  & 138   & 1.08 \\
    \midrule
    FJSP7 & 94    & 0.45  & 94    & 0.17  & 108   & 0.69  & 108   & 0.14  & 100   & 1.86  & 100   & 0.96 \\
    \midrule
    FJSP8 & 178   & 1.72  & 178   & 2.35  & 188   & 2.50   & 188   & 1.91  & 180   & 2.26  & 180   & 2.53 \\
    \midrule
    FJSP9 & 136   & 0.40   & 136   & 0.09  & 148   & 0.51  & 148   & 0.22  & 140   & 1.98  & 140   & 1.79 \\
    \midrule
    FJSP10 & 164   & 0.75  & 164   & 0.71  & 176   & 1.42  & 176   & 2.17  & 166   & 4.16  & 166   & 2.22 \\
    \midrule
    \textbf{Avg} & \textbf{125} & \textbf{0.57} & \textbf{125} & \textbf{0.38} & \textbf{137} & \textbf{0.85} & \textbf{137} & \textbf{0.53} & \textbf{130} & \textbf{1.53} & \textbf{130} & \textbf{1.01} \\
    \bottomrule
    \end{tabular}%
  \label{tab:zone_impact}%
}
\end{table}%

Table~\ref{tab:sensitivity_medium} reports the sensitivity analysis on medium-scale instances, evaluating the impact of zoning and transportation capacity across the four HUdata classes. The table presents average makespans for CP3 and CP4 under four configurations: a base case (2 zones, 2 transbots), increased zoning (4 zones, 4 transbots), increased zone-to-bot ratio (4 zones, 8 transbots), and increased transportation capacity (2 zones, 6 transbots). Numbers in parentheses to the left of each objective value denote the number of instances solved to optimality out of the total number of instances for which a feasible solution was obtained, shown on the right. If no number is shown on the left, then none of the reported objective values were proven optimal. Numbers in parentheses to the right denote the number of instances, out of 40, for which a feasible solution was found within the 600-second time limit. If no number is shown on the right, then feasible solutions were obtained for all instances.

Increasing the number of transbots substantially reduces the makespan across all instance classes. Moving from the base configuration (2 transbots) to higher-capacity settings (6 or 8 transbots) dramatically decreases makespan values, particularly for the edata and rdata classes. For example, the average makespan drops from above 3000 in the base case to nearly half in high-capacity configurations. This confirms that transportation is a major bottleneck in medium-scale instances. Moreover, the number of solved instances also decreases as the system becomes more constrained (e.g., with fewer transbots), indicating that limited transportation capacity not only degrades solution quality but also increases computational difficulty.

Increasing the number of zones introduces a more nuanced effect. While higher zoning levels (e.g., 4 zones) reduce travel distances and can improve makespan when combined with sufficient transportation capacity, they also introduce additional coordination overhead due to inter-zone transfers. This is reflected in the variability of results across configurations: zoning alone does not consistently improve performance, but when coupled with a sufficient number of transbots, it enables significantly better solutions. Additionally, the number of solved instances tends to decrease in highly zoned configurations, suggesting that increased routing constraints and synchronization requirements lead to more challenging search spaces.

\begin{table}[t!]
  \centering
  \caption{Transbot and Zoning Impact Medium Instances}
    \scalebox{0.75}{
    \begin{tabular}{l|r|r|r|r|r|r|r|r}
    \toprule
    Zones & \multicolumn{2}{c|}{2} & \multicolumn{2}{c|}{4} & \multicolumn{2}{c|}{4} & \multicolumn{2}{c}{2} \\
    \midrule
    Transbots & \multicolumn{2}{c|}{2} & \multicolumn{2}{c|}{4} & \multicolumn{2}{c|}{8} & \multicolumn{2}{c}{6} \\
    \midrule
    Instance & CP3 & CP4 & CP3   & CP4   & CP3   & CP4   & CP3   & CP4 \\
    \midrule
    sdata & 4{,}131.95 & 4{,}072.15 & \multicolumn{1}{r|}{(1) 2{,}263.53 (30)} & \multicolumn{1}{r|}{(1) 2{,}794.45} & \multicolumn{1}{r|}{(6) 1{,}262.44 (25)} & \multicolumn{1}{r|}{(7) 1{,}663.48} & \multicolumn{1}{r|}{(9) 1{,}391.70 (30)} & (9) 1{,}658.90 \\
    \midrule
    edata & 3{,}934.08 & 3{,}926.95 & \multicolumn{1}{r|}{2{,}077.43 (30)} & \multicolumn{1}{r|}{2{,}611.98} & \multicolumn{1}{r|}{1{,}291.43 (28)} & \multicolumn{1}{r|}{1{,}593.90} & \multicolumn{1}{r|}{1{,}316.76 (30)} & 1{,}611.28 \\
    \midrule
    rdata & 3{,}179.93 & 3{,}194.95 & \multicolumn{1}{r|}{1{,}570.53 (30)} & \multicolumn{1}{r|}{2{,}128.25} & \multicolumn{1}{r|}{1{,}140.70 (27)} & \multicolumn{1}{r|}{1{,}460.85} & \multicolumn{1}{r|}{1{,}190.93 (30)} & 1{,}427.53 \\
    \midrule
    vdata & 2{,}268.18 & 2{,}915.50 & \multicolumn{1}{r|}{1{,}334.86 (30)} & \multicolumn{1}{r|}{1{,}444.73 (30)} & \multicolumn{1}{r|}{1{,}031.28 (25)} & \multicolumn{1}{r|}{870.69 (16)} & \multicolumn{1}{r|}{808.57 (14)} & 1{,}912.42 (31) \\
    \bottomrule
    \end{tabular}%
    }
  \label{tab:sensitivity_medium}%
\end{table}%

Introducing zoning constraints generally increases the makespan relative to fully flexible configurations, as transbots are restricted to specific regions and inter-zone transfers must be coordinated via handoff points, thereby reducing routing flexibility. Increasing the number of transbots significantly improves performance when transportation resources are scarce by reducing waiting times between operations. However, this effect quickly exhibits diminishing returns, as machine processing constraints eventually dominate the critical path. These results highlight a strong interaction between zoning and transportation capacity. While increasing the number of transbots consistently improves performance, the benefits of zoning depend critically on the available transportation resources. In medium-scale instances, effective system performance requires a careful balance between routing flexibility, congestion management, and transbot availability. It is worth highlighting that CP4 outperforms CP3 in less flexible configurations.

Table~\ref{tab:layout_impact} evaluates the impact of layout design on medium-scale instances by comparing three different layouts, each inducing distinct distance and travel-time matrices. Each layout is randomly generated by sampling distances from a uniform distribution of [20,40], following the approach from \cite{ham2020transfer}. To ensure a fair comparison across the three layout designs, the random seed is varied, generating distinct layout instances for each dataset. The layouts differ in the spatial arrangement of machines, which directly affects transportation times and routing decisions. The average makespans obtained by CP3 and CP4 across the four HUdata classes are reported. 

\begin{table}[t!]
  \centering
  \caption{Layout Impact Medium Instances}
    \begin{tabular}{l|r|r|r|r|r|r}
    \toprule
    Layout & \multicolumn{2}{c|}{Layout 1} & \multicolumn{2}{c|}{Layout 2} & \multicolumn{2}{c}{Layout 3} \\
    \midrule
    Instance & \multicolumn{1}{c|}{CP3} & \multicolumn{1}{c|}{CP4} & \multicolumn{1}{c|}{CP3} & \multicolumn{1}{c|}{CP4} & \multicolumn{1}{c|}{CP3} & \multicolumn{1}{c}{CP4} \\
    \midrule
    sdata & 4{,}131.95 & 4{,}072.15 & 4{,}176.08 & 4{,}089.98 & 4{,}215.98 & 4{,}124.28 \\
    \midrule
    edata & 3{,}934.08 & 3{,}926.95 & 3{,}927.45 & 3{,}942.15 & 4{,}005.95 & 3{,}966.73 \\
    \midrule
    rdata & 3{,}179.93 & 3{,}194.95 & 3{,}150.63 & 3{,}183.30 & 3{,}187.88 & 3{,}186.65 \\
    \midrule
    vdata & 2{,}268.18 & 2{,}915.50 & 2{,}116.13 & 2{,}617.20 & 2{,}230.83 & 2{,}701.53 \\
    \midrule
    \textbf{Avg} & \textbf{3{,}378.53} & \textbf{3{,}527.39} & \textbf{3{,}342.57} & \textbf{3{,}458.16} & \textbf{3{,}410.16} & \textbf{3{,}494.79} \\
    \bottomrule
    \end{tabular}%
  \label{tab:layout_impact}%
\end{table}%

The results show that layout design has a noticeable impact on system performance, although its effect is more moderate compared to zoning and transportation capacity. Layout~2 consistently yields the best average performance for both formulations, reducing the makespan relative to Layout~1, while Layout~3 slightly degrades performance. The underlying travel-time structure drives these differences: more favorable layouts reduce average transportation distances and improve coordination between machines and transbots, leading to shorter overall schedules.

The impact of layout is particularly pronounced for the vdata class, with significant variations in makespan across layouts. This suggests that instances with higher routing flexibility or variability are more sensitive to spatial configuration. In contrast, the edata and rdata classes exhibit more stable behavior, indicating that their performance is less dependent on layout structure and more driven by inherent processing and routing constraints.

Overall, these results highlight that layout design influences scheduling performance through its effect on travel times, but its impact remains secondary to key structural factors such as transportation capacity and zoning. Nonetheless, selecting an appropriate layout can yield meaningful improvements, particularly when transportation plays a critical role in the system.  

\subsection{Managerial Insights}

The FJSPT-H model offers substantial operational and strategic benefits for manufacturing facilities integrating mobile transfer transbots. By explicitly capturing heterogeneous transbots, zoning constraints, and handoff points, the model enables managers to make informed decisions regarding machine allocation, transbot assignment, and workflow sequencing. Zones and handoff points not only ensure collision-free operations but also enable predictable and repeatable material flow, which is critical in high-mix, high-throughput environments where delays or conflicts can cascade across the shop floor. Incorporating these factors into the planning stage reduces reliance on ad hoc adjustments, increases system reliability, and supports scalable operations.

Zoning and handoff strategies provide a structured framework for preventing congestion and optimizing spatial layout. Assigning transbots to specific zones ensures that each transbot operates within its capacity and accessibility constraints, while handoff points coordinate transfers between zones efficiently. This allows managers to design shop floors that balance workloads across machines, minimize queuing at stations, and reduce idle times for both transbots and machines. By modeling these operational details explicitly, managers can experiment with alternative floor layouts and determine the most efficient assignment of transbots to zones, ultimately enhancing overall productivity.

Mixed manufacturing systems with heterogeneous transbots demonstrate increased robustness and flexibility. Even when transbots differ in load capacity, speed, and zone accessibility, the model shows that effective scheduling and routing can significantly reduce makespan and improve throughput. This flexibility enables facilities to absorb variability in job sizes, machine availability, and unplanned maintenance, without degrading system performance. Managers can use this insight to invest strategically in transbot diversity, ensuring that the fleet can handle varying operational requirements while maintaining high service levels.

The structural differences between the two formulations suggest distinct advantages depending on the operational environment. The arc-based formulation provides a more modular representation by separating arc selection, leg decomposition, and transbot assignment. This structure is particularly advantageous in environments with complex routing logic, strict zoning policies, or frequent inter-zone transfers, where explicit control over routing decisions enhances interpretability and facilitates incorporating additional constraints. In contrast, the operation-embedded formulation integrates routing and processing decisions into a single modeling layer. This tighter coupling can enhance propagation and reduce redundant decision layers, thereby improving computational performance in environments with high machine flexibility or dense routing networks. Because it avoids an explicit arc-selection layer, the operation-embedded model may also exhibit a smaller memory footprint when the number of feasible arcs per operation is large. Therefore, the arc-based formulation favors transparency and modular routing control, whereas the operation-embedded formulation emphasizes structural compactness and integration. Instance characteristics, such as zoning density, routing complexity, and machine flexibility, should guide the choice between the two.

The CP framework itself provides a practical decision-support tool. By capturing machine flexibility, tansbot heterogeneity, and temporal constraints in a single model, managers can explore trade-offs between transbot utilization, transfer times, and job completion times. This enables quantitative assessment of operational policies, such as prioritizing high-value jobs, reallocating transbots in real-time, or evaluating the impact of additional handoff points. The CP-based approach also supports rapid experimentation with different scenarios, which is particularly valuable in dynamic or seasonal production environments where traditional scheduling methods may be too rigid or slow to adapt.

Finally, the quick solution to benchmark-inspired and case-based instances adds a layer of reproducibility to operational planning. Organizations can use these layouts to validate new scheduling strategies, compare different algorithmic approaches, and refine operational policies under realistic conditions. It also allows managers to identify potential bottlenecks, calibrate transbot allocation rules, and improve predictive maintenance planning. Overall, the FJSPT-H model supports evidence-based decision-making, enhancing efficiency, safety, and responsiveness in automated production systems.

\section{Conclusion}\label{section:8}

This study addresses the coordinated scheduling of operations and material transfers in flexible job shops equipped with heterogeneous transbots. By integrating job sequencing and transbot routing into a single framework, the research formulates a complex, NP-hard problem that extends the classical flexible job shop problem by incorporating realistic transbot routing considerations. The proposed constraint programming models effectively capture machine flexibility, transbot heterogeneity, load capacities, and collision-free path planning. The introduction of a book-and-release strategy enables transbots to operate efficiently without rigid routing, improving practical applicability. Computational experiments on benchmark instances demonstrate that the approach produces high-quality solutions with reasonable computation times. The publicly released instances provide a foundation for future research on joint scheduling and routing in automated manufacturing environments. Collectively, the results underscore the potential of constraint programming to support flexible, high-throughput, and robust production systems in real-world manufacturing settings.

\section*{Acknowledgments}

This research was partly supported by the NSF AI Institute for Advances in Optimization (Award 2112533) and the Georgia Tech Manufacturing Institute (GTAIM).

\begin{spacing}{1}
\typeout{}
\bibliographystyle{apalike}
\bibliography{References.bib}

@inproceedings{booth2016constraint,
  title={A constraint programming approach to multi-robot task allocation and scheduling in retirement homes},
  author={Booth, Kyle EC and Nejat, Goldie and Beck, J Christopher},
  booktitle={International conference on principles and practice of constraint programming},
  pages={539--555},
  year={2016},
  organization={Springer}
}

@article{homayouni2021production,
  title={Production and transport scheduling in flexible job shop manufacturing systems},
  author={Homayouni, Seyed Mahdi and Fontes, Dalila BMM},
  journal={Journal of Global Optimization},
  volume={79},
  number={2},
  pages={463--502},
  year={2021},
  publisher={Springer}
}

@article{yan2021research,
  title={Research on flexible job shop scheduling under finite transportation conditions for digital twin workshop},
  author={Yan, Jun and Liu, Zhifeng and Zhang, Caixia and Zhang, Tao and Zhang, Yueze and Yang, Congbin},
  journal={Robotics and Computer-Integrated Manufacturing},
  volume={72},
  pages={102198},
  year={2021},
  publisher={Elsevier}
}

@article{homayouni2023multistart,
  title={A multistart biased random key genetic algorithm for the flexible job shop scheduling problem with transportation},
  author={Homayouni, S Mahdi and Fontes, Dalila BMM and Gon{\c{c}}alves, Jos{\'e} F},
  journal={International Transactions in Operational Research},
  volume={30},
  number={2},
  pages={688--716},
  year={2023},
  publisher={Wiley Online Library}
}

@article{chambers1996new,
  title={New tabu search results for the job shop scheduling problem},
  author={Chambers, John B and Barnes, J Wesley},
  journal={The University of Texas, Austin, Technical Report Series ORP96-06, Graduate Program in Operations Research and Industrial Engineering},
  year={1996},
  publisher={Citeseer}
}

@article{kumar2011simultaneous,
  title={Simultaneous scheduling of machines and vehicles in an FMS environment with alternative routing},
  author={Kumar, MV Satish and Janardhana, Ranga and Rao, CSP},
  journal={The International Journal of Advanced Manufacturing Technology},
  volume={53},
  number={1},
  pages={339--351},
  year={2011},
  publisher={Springer}
}

@inproceedings{deroussi2014hybrid,
  title={A hybrid PSO applied to the flexible job shop with transport},
  author={Deroussi, Laurent},
  booktitle={International conference on swarm intelligence based optimization},
  pages={115--122},
  year={2014},
  organization={Springer}
}

@ARTICLE{11339952,
  author={Yao, Youjie and Liu, Qihao and Li, Xinyu and Gao, Liang},
  journal={IEEE Transactions on Automation Science and Engineering}, 
  title={Constraint Programming for AGV and Machine Integrated Scheduling Problem in Flexible Manufacturing System}, 
  year={2026},
  volume={23},
  number={},
  pages={2378-2390},
  doi={10.1109/TASE.2025.3650678}}

@article{yao2024novel,
  title={A novel mathematical model for the flexible job-shop scheduling problem with limited automated guided vehicles},
  author={Yao, Youjie and Liu, Qihao and Fu, Ling and Li, Xinyu and Yu, Yanbin and Gao, Liang and Zhou, Wei},
  journal={IEEE Transactions on Automation Science and Engineering},
  year={2024},
  publisher={IEEE}
}

@article{yao2026knowledge,
  title={A knowledge-enhanced discrete artificial bee colony algorithm for flexible job shop scheduling problem with transport robots},
  author={Yao, Youjie and Ye, Kai and Zhang, Chunjiang and Li, Xinyu and Gao, Liang},
  journal={Journal of Manufacturing Systems},
  volume={84},
  pages={614--628},
  year={2026},
  publisher={Elsevier}
}

@article{ham2020transfer,
  title={Transfer-robot task scheduling in flexible job shop},
  author={Ham, Andy},
  journal={Journal of Intelligent Manufacturing},
  volume={31},
  number={7},
  pages={1783--1793},
  year={2020},
  publisher={Springer}
}

@inproceedings{deroussi2010simultaneous,
  title={Simultaneous scheduling of machines and vehicles for the flexible job shop problem},
  author={Deroussi, L and Norre, S},
  booktitle={International conference on metaheuristics and nature inspired computing},
  pages={1--2},
  year={2010},
  organization={Djerba Island Tunisia}
}

@article{berterottiere2024flexible,
  title={Flexible job-shop scheduling with transportation resources},
  author={Berterotti{\`e}re, Lucas and Dauz{\`e}re-P{\'e}r{\`e}s, St{\'e}phane and Yugma, Claude},
  journal={European Journal of Operational Research},
  volume={312},
  number={3},
  pages={890--909},
  year={2024},
  publisher={Elsevier}
}

@article{naderi2023mixed,
  title={Mixed-integer programming vs. constraint programming for shop scheduling problems: new results and outlook},
  author={Naderi, Bahman and Ruiz, Rub{\'e}n and Roshanaei, Vahid},
  journal={INFORMS Journal on Computing},
  volume={35},
  number={4},
  pages={817--843},
  year={2023},
  publisher={INFORMS}
}

@article{ku2016mixed,
  title={Mixed integer programming models for job shop scheduling: A computational analysis},
  author={Ku, Wen-Yang and Beck, J Christopher},
  journal={Computers \& Operations Research},
  volume={73},
  pages={165--173},
  year={2016},
  publisher={Elsevier}
}

@article{bilge1995time,
  title={A time window approach to simultaneous scheduling of machines and material handling system in an FMS},
  author={Bilge, {\"U}mit and Ulusoy, G{\"u}nd{\"u}z},
  journal={Operations Research},
  volume={43},
  number={6},
  pages={1058--1070},
  year={1995},
  publisher={INFORMS}
}

@article{hurink1994tabu,
  title={Tabu search for the job-shop scheduling problem with multi-purpose machines},
  author={Hurink, Johann and Jurisch, Bernd and Thole, Monika},
  journal={Operations-Research-Spektrum},
  volume={15},
  number={4},
  pages={205--215},
  year={1994},
  publisher={Springer}
}

@inproceedings{raman1986simultaneous,
  title={Simultaneous scheduling of machines and material handling devices in automated manufacturing},
  author={Raman, N},
  booktitle={Proc. of the Second ORSA/TIMS Conference on Flexible Manufacturing Systems: Operations Research Models and Applications, 1986},
  year={1986}
}

@article{hurink2005tabu,
  title={Tabu search algorithms for job-shop problems with a single transport robot},
  author={Hurink, Johann and Knust, Sigrid},
  journal={European journal of operational research},
  volume={162},
  number={1},
  pages={99--111},
  year={2005},
  publisher={Elsevier}
}

@article{karimi2017scheduling,
  title={Scheduling flexible job-shops with transportation times: Mathematical models and a hybrid imperialist competitive algorithm},
  author={Karimi, Sajad and Ardalan, Zaniar and Naderi, B and Mohammadi, M},
  journal={Applied mathematical modell{\'\i}ng},
  volume={41},
  pages={667--682},
  year={2017},
  publisher={Elsevier}
}

@article{nouri2016simultaneous,
  title={Simultaneous scheduling of machines and transport robots in flexible job shop environment using hybrid metaheuristics based on clustered holonic multiagent model},
  author={Nouri, Houssem Eddine and Driss, Olfa Belkahla and Gh{\'e}dira, Khaled},
  journal={Computers \& Industrial Engineering},
  volume={102},
  pages={488--501},
  year={2016},
  publisher={Elsevier}
}

@article{he2025hybrid,
  title={A hybrid approach using ant colony optimisation for integrated scheduling of production and transportation tasks within flexible manufacturing systems},
  author={He, Naihui and Zhang, David and Bettayeb, Belgacem and others},
  journal={M'hammed and Zhang, David and Bettayeb, Belgacem, A Hybrid Approach Using Ant Colony Optimisation for Integrated Scheduling of Production and Transportation Tasks within Flexible Manufacturing Systems},
  year={2025}
}

@misc{Dresser2025AmazonRobotics,
  author       = {Scott Dresser},
  title        = {Amazon launches a new AI foundation model to power its robotic fleet and deploys its 1 millionth robot},
  howpublished = {About Amazon (official Amazon news site)},
  month        = jun,
  year         = {2025},
  day          = {30},
  url          = {https://www.aboutamazon.com/news/operations/amazon-million-robots-ai-foundation-model}
}

@misc{Jenkins2025WarehouseRobotics,
  author       = {Abby Jenkins},
  title        = {What Is Warehouse Robotics? The Ultimate Guide for 2025},
  howpublished = {NetSuite (Resource Center, article)},
  month        = may,
  year         = {2025},
  day          = {22},
  note         = {Online; accessed on \today},
  url          = {https://www.netsuite.com/portal/resource/articles/ecommerce/warehouse-robotics.shtml}
}

@misc{Repko2023WalmartAutomation,
  author       = {Melissa Repko},
  title        = {Walmart chases higher profits powered by warehouse robots and automated claws},
  howpublished = {CNBC},
  month        = apr,
  year         = {2023},
  day          = {11},
  note         = {Online; accessed on \today},
  url          = {https://www.cnbc.com/2023/04/11/walmart-warehouse-automation-powers-higher-profits.html}
}

@article{dauzere2024flexible,
  title={The flexible job shop scheduling problem: A review},
  author={Dauz{\`e}re-P{\'e}r{\`e}s, St{\'e}phane and Ding, Junwen and Shen, Liji and Tamssaouet, Karim},
  journal={European Journal of Operational Research},
  volume={314},
  number={2},
  pages={409--432},
  year={2024},
  publisher={Elsevier}
}

@article{xiong2022survey,
  title={A survey of job shop scheduling problem: The types and models},
  author={Xiong, Hegen and Shi, Shuangyuan and Ren, Danni and Hu, Jinjin},
  journal={Computers \& Operations Research},
  volume={142},
  pages={105731},
  year={2022},
  publisher={Elsevier}
}

@article{mo2024capacitated,
  title={Capacitated Vehicle Routing Problem with Pickup and Delivery in Robotic Mobile Fulfillment Systems},
  author={Mo, Ni-Lei and Zhang, Wencong},
  journal={IEEE Access},
  year={2024},
  publisher={IEEE}
}

@article{brandimarte1993routing,
  title={Routing and scheduling in a flexible job shop by tabu search},
  author={Brandimarte, Paolo},
  journal={Annals of Operations research},
  volume={41},
  number={3},
  pages={157--183},
  year={1993},
  publisher={Springer}
}

@article{lenstra1981complexity,
  title={Complexity of vehicle routing and scheduling problems},
  author={Lenstra, Jan Karel and Kan, AHG Rinnooy},
  journal={Networks},
  volume={11},
  number={2},
  pages={221--227},
  year={1981},
  publisher={Wiley Online Library}
}

@article{hooker2002logic,
  title={Logic, optimization, and constraint programming},
  author={Hooker, John N},
  journal={INFORMS Journal on Computing},
  volume={14},
  number={4},
  pages={295--321},
  year={2002},
  publisher={INFORMS}
}

@article{zhang2019review,
  title={Review of job shop scheduling research and its new perspectives under Industry 4.0},
  author={Zhang, Jian and Ding, Guofu and Zou, Yisheng and Qin, Shengfeng and Fu, Jianlin},
  journal={Journal of intelligent manufacturing},
  volume={30},
  number={4},
  pages={1809--1830},
  year={2019},
  publisher={Springer}
}

@article{xin2025review,
  title={A review of flexible job shop scheduling problems considering transportation vehicles},
  author={Xin, Bin and Lu, Sai and Wang, Qing and Deng, Fang},
  journal={Frontiers of Information Technology \& Electronic Engineering},
  volume={26},
  number={3},
  pages={332--353},
  year={2025},
  publisher={Springer}
}

@article{destouet2023flexible,
  title={Flexible job shop scheduling problem under Industry 5.0: A survey on human reintegration, environmental consideration and resilience improvement},
  author={Destouet, Candice and Tlahig, Houda and Bettayeb, Belgacem and Mazari, B{\'e}lahc{\`e}ne},
  journal={Journal of Manufacturing Systems},
  volume={67},
  pages={155--173},
  year={2023},
  publisher={Elsevier}
}

@article{kusiak2019intelligent,
  title={Intelligent manufacturing: bridging two centuries},
  author={Kusiak, Andrew},
  journal={Journal of intelligent manufacturing},
  volume={30},
  number={1},
  pages={1--2},
  year={2019},
  publisher={Springer}
}
\end{spacing}    

\begin{appendices}
\renewcommand{\thesection}{\Alph{section}}
\section{CP Functions} 
\label{A}

Some CP-related functions must be defined as they are used in the formulation.

\begin{itemize}
\item \textit{Interval Variables:} a type of decision variable representing a time interval during which an activity occurs. In the context of the FJSPT-H, it could be an operation occurring on a machine or a transbot's material transfer. The start time $s$, end time $e$, and presence of the interval are the three essential characteristics. The domain for presence is either 1 if the task was executed, or 0 otherwise. It can only be 0 if the variable is explicitly declared as \textit{optional}, meaning the execution of a task is not required across the set of all intervals. Formally, the domain for an interval is defined as a subset of $\{\perp\} \cup \{[s,e] \mid s,e \in \mathbb{Z},\ s \leq e\}$.

\item \textit{Interval Sequence Variables:} another type of decision variable defined on a set of interval variables responsible for representing its ordering. Any intuitive constraints to define conditions for overlapping or sequence must be explicitly declared through complementary precedence constraints. These variables are used to sort tasks on machines and to order machines for each transbot to service.  

\item \textit{Alternative Constraints:} a type of global constraint used in the assignment of interval variables. Assume $a$ is an interval variable and and $B$ is a list of intervals where $B=[b_1,...,b_n]$. Then, \textit{alternative($a, [b_1,...,b_n]$)} states that if interval $a$ is present, then exactly one interval from $B$ is also present and the chosen variable from $B$ determines the start and end time for $a$.

\item \textit{Precedence Constraints:} a type of global constraint used to examine interval sequence variables. This type specifies the relative positions of variables within the solution and constrains variable ordering. These functions are often self-explanatory, such as \textit{endBeforeStart()} and \textit{noOverlap()}, and determine the start and end of an interval relative to another interval.  
\end{itemize}

\section{Instance Generation} 
\label{B}

The small-scale instances used in this paper were adapted from \cite{deroussi2010simultaneous}, which extended the original \cite{bilge1995time} instances by considering a flexible environment with machine alternatives. To enable zoned experiments for computational results and sensitivity analysis, the original layout matrix is modified by introducing one additional machine representing the handoff point, with its associated travel times sampled from a uniform distribution over [2,8]. As described previously, the resultant layouts account for the aggregation of all legs in the travel times between machines using the distance from the handoff point for intra-zone transfers. For each zoned experiment, machines are assigned to zones in a cyclic sequence, with zones labeled from 1 to the total number of zones, applied consecutively across machines in index order. Transbots are assigned to zones analogously.

The medium-scale instances implement the \textit{Layout 1} generation from \cite{ham2020transfer}, which sampled travel times from a uniform distribution over [20,40] for the benchmark instances from \cite{hurink1994tabu}. For each instance, a distinct layout matrix and sample travel times to the handoff point from the same distribution are generated, ensuring consistency in the underlying spatial assumptions. Zone assignments for both machines and transbots follow the same cyclic assignment as in the small-scale setting. Experiments are restricted to scenarios where there are at least as many transbots as zones. Under this assignment pattern, each zone contains on average $\frac{|\mathcal{M}|}{|\mathcal{Z}|}$ machines and $\frac{|\mathcal{V}|}{|\mathcal{Z}|}$ transbots.

\end{appendices}

\end{document}